\documentclass{article}
\usepackage[utf8]{inputenc}
\usepackage{enumitem}  
\usepackage[T1]{fontenc}
\usepackage{amsmath,amsthm,amsfonts,amssymb} 
\usepackage{bbm} 
\usepackage{graphicx} 
\usepackage{float} 
\usepackage{hyperref} 
\usepackage{xcolor} 
\usepackage{mathrsfs}
\hypersetup{colorlinks=true,citecolor={black},
    linkbordercolor={white},
    linkcolor={black},urlcolor={black}} 
\usepackage{tikz} 
\usepackage{authblk}

\usepackage{geometry} 
\newtheorem{theo}{Theorem}[section]
\newtheorem*{theo*}{Theorem}
\newtheorem{prop}[theo]{Proposition}
\newtheorem*{prop*}{Proposition}
\newtheorem{lemme}[theo]{Lemma}
\newtheorem*{lemme*}{Lemma}
\newtheorem{definition}[theo]{Definition}
\newtheorem{corollaire}[theo]{Corollary}
\newtheorem{exemple}[theo]{Example}
\newtheorem{remarque}{Remark}

\newcommand{\PP}{\mathbb{P}}
\newcommand{\EE}{\mathbb{E}}
\newcommand{\NN}{\mathbb{N}}
\newcommand{\CC}{\mathbb{C}}
\newcommand{\ZZ}{\mathbb{Z}}
\newcommand{\RR}{\mathbb{R}}
\newcommand{\XX}{\mathbb{X}}
\newcommand{\Var}{\mathrm{Var}}
\newcommand{\tend}[4][]{\overset{#4}{\xrightarrow[#2\to#3]{#1}}}

\title{A central limit theorem for the variation of the sum of digits}

\author[2]{Yohan HOSTEN
}

\author[2]{Élise JANVRESSE
}
\author[1]{Thierry de la RUE
}

\affil[1]{LMRS, Université de Rouen Normandie, CNRS UMR 6085, Avenue de l'Université, 76801, Saint Étienne du Rouvray, France, \href{mailto:Thierry.de-la-Rue@univ-rouen.fr}{Thierry.de-la-Rue@univ-rouen.fr}}
\affil[2]{LAMFA, Université de Picardie Jules Verne, CNRS UMR 7352, 33, rue Saint-Leu, 80000, Amiens, France, \href{mailto:elise.janvresse@u-picardie.fr}{elise.janvresse@u-picardie.fr} and \href{mailto:yohan.hosten@u-picardie.fr}{yohan.hosten@u-picardie.fr}}
\date{}
\begin{document}

\maketitle


\textbf{Abstract.} We prove a Central Limit Theorem for probability measures defined via the variation of the sum-of-digits function, in base $b\ge 2$. For $r\ge 0$ and $d\in\ZZ$, we consider $\mu^{(r)}(d)$ as the density of integers $n\in\NN$ for which the sum of digits increases by $d$ when we add $r$ to $n$. We give a probabilistic interpretation of $\mu^{(r)}$ on the probability space given by the group of $b$-adic integers equipped with the normalized Haar measure. We split the base-$b$ expansion of the integer $r$ into so-called ``blocks'', and we consider the asymptotic behaviour of $\mu^{(r)}$ as the number of blocks goes to infinity. We show that, up to renormalization, $\mu^{(r)}$ converges to the standard normal law as the number of blocks of $r$ grows to infinity. We provide an estimate of the speed of convergence. The proof relies, in particular, on a $\phi$-mixing process defined on the $b$-adic integers.

\textbf{Résumé.} On prouve un  Théorème Central Limite pour des mesures de probabilités définies grâce à la variation de la somme des chiffres en base $b\ge 2$. Pour $r\ge 0$ et $d\in\ZZ$, on considère $\mu^{(r)}(d)$, la densité des entiers $n\in\NN$ pour lesquels la somme des chiffres augmente de $d$ quand on ajoute $r$ à $n$. On donne une interprétation probabiliste de $\mu^{(r)}$ sur l'espace de probabilités donné par le groupe des entiers $b$-adiques muni de la mesure de Haar renormalisée. On décompose l'écriture en base $b$ d'un entier $r$ en ce que l'on appelle des ``blocs'', et nous considérons le comportement asymptotique de $\mu^{(r)}$ quand le nombre de blocs tend vers l'infini. On montre qu'à renormalisation près, $\mu^{(r)}$ converge vers une loi normale centrée réduite quand le nombre de blocs de $r$ tend vers l'infini. Nous fournissons une estimation de la vitesse de convergence. La preuve repose, entre autre, sur un processus $\phi$-mélangeant défini sur les entiers $b$-adiques.

\vspace{5mm}

\textbf{Keywords.} Sum of digits, Central Limit Theorem, $b$-adic odometer, $\phi$-mixing.

\vspace{5mm}

\textbf{MSC (2020).} 11A63, 37A44 and 60F05.

\section{Introduction}
Throughout this article, \emph{integers} mean elements of the set $\NN:=\{0,1,2, \cdots\}$, and $b$ denotes a fixed integer, $b\ge 2$.

For an integer $n$, we consider the associated sequence of \emph{digits} $(n_k)\in\{0,\cdots, b-1\}^{\NN}$, finitely many of them being strictly positive, such that
\[n=\sum_{k\ge 0}n_kb^k.\] 
For $n\neq 0$ and $\ell:=\max\{k: n_k\neq 0\}$, we introduce the notation $\overline{n_{\ell}\cdots n_0}:=n$, which generalises the usual way we write numbers in base $10$, and which we refer to as the (\emph{base $b$}) \emph{expansion} of $n$. By convention, we set $\overline{0}:=0$. 
Then, we define the \emph{sum-of-digits} function (in base $b$) as 
\[s(n):=\sum_{k\ge 0} n_k.\]

A central object in our paper is the variation of the sum of digits when we add a fixed integer $r$ to $n$: for $r,n\in\NN$, we set
\begin{align}\label{refDeltar}
    \Delta^{(r)}(n)&:=s(n+r)-s(n).
\end{align}
An interesting feature of $\Delta^{(r)}$ is that it gives the number of carries created during the addition $n+r$ in base $b$. To be more precise, if $c$ is the number of carries then
\begin{align}\label{lien_delta_retenue_entier}
    \Delta^{(r)}(n)&=s(r)-c (b-1).
\end{align}
Bésineau \cite{JB} proves that, for every $d\in\ZZ$, the following asymptotic density exists
\[\mu^{(r)}(d):=\lim\limits_{N\rightarrow +\infty}\frac{1}{N}\Big|\Big\{n<N: \Delta^{(r)}(n)=d\Big\}\Big|\]
and he studies these asymptotic densities through their correlation function. Actually, since $\sum_{d\in\ZZ} \mu^{(r)}(d)=1$, the function $\mu^{(r)}$ defines a probability measure on $\ZZ$.
 
Morgenbesser and Spiegelhofer \cite{JFMLS} show an amazing property: the measure $\mu^{(r)}$ remains the same if we reverse the order of the digits in the expansion of $r$. They call it the \emph{reverse property}.

In the particular case $b=2$, Emme and Prikhod'ko \cite{JEAP} show that the variance of $\mu^{(r)}$ is bounded from above and below by a constant multiplied by the number of blocks of $1$'s in the binary expansion of $r$. In Section \ref{Variance of Delta}, we extend this result to each $b\ge 2$.

Also in the binary case, Emme and Hubert \cite{JEPH1} show that, for almost every sequence of integers $(r_n)_{n\in\NN}$ written in binary and defined via a balanced Bernoulli process, the sequence of measures $(\mu^{(r_n)})_{n\in\NN}$,  after renormalization, converges in distribution to the standard normal law. The proof is done by computing all the moments of $\mu^{(r_n)}$ and by showing that, after renormalization,  they converge to the moments of the standard normal law. In our paper, we prove a more accurate and more general Central Limit Theorem (CLT).

To do so, we study the variations of the sum of digits in the context of an appropriate probability space. We consider the compact additive group $(\XX,+)$ of \emph{$b$-adic integers}. The space $\XX$ is endowed with the Borel $\sigma$-algebra and its normalized Haar measure $\PP$. 

We extend $\Delta^{(r)}$ almost everywhere on $\XX$ and show in Section \ref{Odometer and sum-of-digits function} (Proposition \ref{prop_moment}) that, for every $d\in\ZZ$ 
\[\mu^{(r)}(d)=\PP\left(\{x\in \XX: \Delta^{(r)}(x)=d \}\right).\]

To state our main result, we have to define the notion of \emph{blocks} in the base $b$ expansion of an integer $r$.
\begin{definition}\label{defblock}
   A \emph{block} in the expansion $\overline{r_{\ell}\cdots r_0}$ of an integer $r\in\NN$ is defined as follows: it is either
   \begin{enumerate}
       \item a maximal sequence of consecutive digits equal to $0$ (``block of $0$'s'') or
       \item a maximal sequence of consecutive digits equal to $b-1$ (``block of $(b-1)$'s'')  or
       \item when $b\ge 3$, a digit between $1$ and $b-2$ (``single-digit block'').
   \end{enumerate}
   We also define the quantity $\rho(r)$ as the number of blocks in the base $b$ expansion of $r$. 
\end{definition}
\begin{figure}[H]
    \centering
    \includegraphics[scale=0.35]{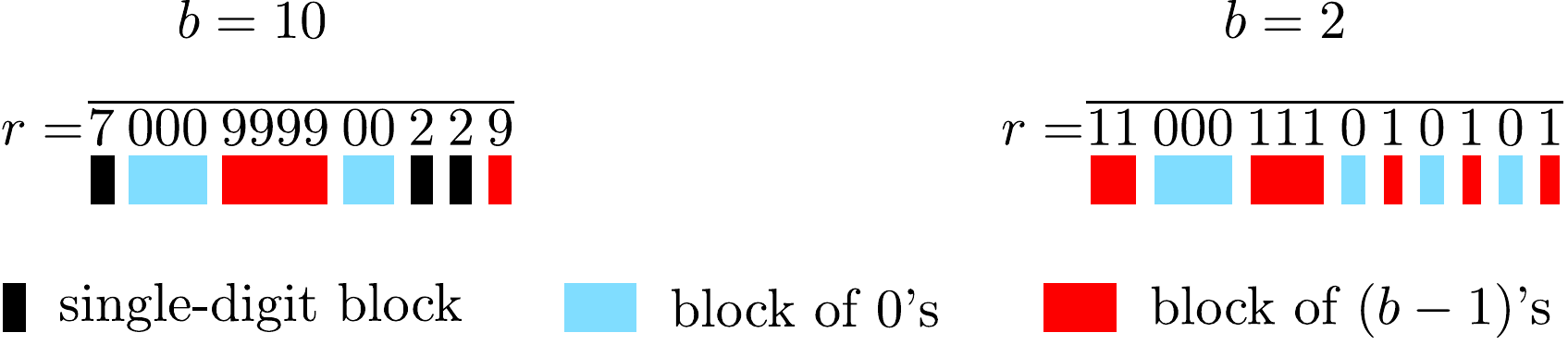}
    \caption{Examples of the decomposition in blocks in decimal and binary bases.  On the left-hand side, $\rho(r)=7$. On the right-hand side, $\rho(r)=9$.}
    \label{blocktranslation}
\end{figure}

We specify that a block of $0$'s or of $(b-1)$'s of length $1$ is not considered, in this paper, as a single-digit block. 

The following theorem, which generalizes Emme and Prikhod'ko's result, states that the number of blocks $\rho(r)$ controls the variance of $\mu^{(r)}$.
\begin{theo}\label{encadrement_de_la_variance}
For every integer $r\ge 1$ 
\[\frac{b}{4}\rho(r)\le \Var(\mu^{(r)})\le 2b^2 \rho(r).\]
\end{theo}
Now,  we need to introduce, for an integer $r\ge 1$, the standard deviation $\sigma_r:=\sqrt{\Var(\mu^{(r)}})>0$ and the renormalized measure $\widetilde{\mu}^{(r)}$ which is the measure on $\RR$ concentrated on the points of the form $\frac{d}{\sigma_r}$ ($d\in\ZZ$) and which satisfies
\[\forall d\in\ZZ, \hspace{5mm} \widetilde{\mu}^{(r)}\hspace{-1mm}\left(\frac{d}{\sigma_r}\right):=\mu^{(r)}(d).\]
Our main result states that, for an integer $r\ge 1$, the renormalized measure $\widetilde{\mu}^{(r)}$ converges in distribution to the standard normal law as the number of blocks tends to infinity. 
\begin{theo}\label{TCL}
We have the convergence
\[\widetilde{\mu}^{(r)}\tend{\rho(r)}{+\infty}{d}{\mathcal{N}(0,1)}.\]
\end{theo}
Theorem~\ref{TCL} can be seen as a direct consequence of the following theorem, which furthermore provides an estimation of the speed of convergence.
\begin{theo}\label{theoprincipal}
Let $h:\RR\rightarrow\RR$ be a thrice differentiable function with $||h'''||_{\infty}<\infty$. Let $Z$ be a random variable following $\widetilde{\mu}^{(r)}$ and $Y$ a standard normal random variable. Then
\begin{align}\label{vitesse_fonction_reguliere}
\bigg|\EE(h(Z))-\EE(h(Y))\bigg|&= \underset{\rho(r)\rightarrow\infty}{O}\left(\frac{1}{\sqrt{\rho(r)}}\right).
\end{align}
Furthermore, if we denote by $F_r$ (respectively $F$) the cumulative distribution function of $\widetilde{\mu}^{(r)}$ (respectively $\mathcal{N}(0,1)$), then there exists $\widetilde{K}>0$ such that for every integer $r\ge 1$
\begin{align}\label{vitesse indicatrice}
    \sup_{t\in\RR}\left|F_r(t)-F(t)\right|&\le \frac{\widetilde{K}}{\rho(r)^{\frac{1}{8}}}.
\end{align}
\end{theo}
A result in the same spirit has recently been published by Spiegelhofer and Wallner \cite{LSMW}. In the case of the base $2$, they give a very accurate estimation of the measure $\mu^{(r)}(d)$ for every $d\in\ZZ$ 
\begin{align}\label{result SW}
    \mu^{(r)}(d)&=\frac{1}{\sigma_r\sqrt{2\pi}}e^{-\frac{d^2}{2\sigma_r^2}}+\underset{\rho(r)\rightarrow\infty}{O}\left(\rho(r)^{-1}(\log(\rho(r))^4\right).
\end{align}
Their result is proved using a combination of several techniques such as recurrence relations, cumulant generating functions, and integral representations. It seems possible but extremely difficult to generalize \eqref{result SW} to other bases. It also implies a CLT when $\rho(r)$ tends to infinity: using \eqref{result SW}, it is possible to show that for every real numbers $a<b$
\[\widetilde{\mu}^{(r)}([a,b])\tend{\rho(r)}{\infty}{}{\frac{1}{\sqrt{2\pi}}\int_{a}^b e^{\frac{-t^2}{2}}\mathrm{d}t}\]
with a speed of convergence of $\frac{\log^4(\rho(r))}{\rho(r)^{\frac{1}{2}}}$. However, it is not clear how we can get a speed of convergence of $F_r(t)$ to $F(t)$. On our side, we use a drastically different approach which applies directly in any base, relying on the concept of \emph{$\phi$-mixing} process and on a result from Sunklodas \cite{JKS}.

\subsection*{Roadmap}

Section \ref{Odometer and sum-of-digits function} is devoted to placing the study of the measures $\mu^{(r)}$ in the context of the odometer on the set $\XX$ of $b$-adic integers. We extend $\Delta^{(r)}$ almost everywhere on $\XX$ and we show that the convergence
\[\lim\limits_{N\to\infty}\frac{1}{N}\sum_{n<N}f(\Delta^{(r)}(n))=\int_{\XX} f(\Delta^{(r)}(x)) \mathrm{d}\PP(x)\]
(where $\PP$ is the normalized Haar measure on $\XX$) is satisfied for functions $f:\ZZ\rightarrow \CC$ of polynomial growth (Proposition \ref{prop_moment}) and, more generally, for functions $f$ such that $f\circ\Delta^{(r)}$ is integrable (Proposition \ref{prop_gene}). We deduce from Proposition \ref{prop_moment} that $\mu^{(r)}(d)=\PP(\{x\in \XX: \Delta^{(r)}(x)=d \})$ and that $\mu^{(r)}$ has finite moments. In particular, we show that $\mu^{(r)}$ is of zero-mean.

In Section \ref{Variance of Delta}, we focus on the second moment of $\mu^{(r)}$. First, we establish in Proposition~\ref{relat_rec_mesure} an inductive relation between the measures in the spirit of Bésineau's result \cite[p.13]{JB}. From that, we deduce an inductive relation on the variance of the measures (Lemma \ref{lemme_relat_rec_var}). 
Then, we prove the estimation of the variance stated in Theorem~\ref{encadrement_de_la_variance}.

In Section \ref{A phi-mixing process}, we build a finite sequence of random variables associated to the addition of some integer $r$ that will be used to prove Theorem~\ref{theoprincipal}. We estimate the $\phi$-mixing coefficients for this sequence.

The last section is devoted to the proof of Theorem~\ref{theoprincipal}. We show how we can apply a result from Sunklodas \cite[Theorem 1]{JKS} giving a speed of convergence in the CLT for $\phi$-mixing process.

\subsection*{Acknowledgment} 
We thank M. El Machkouri for fruithful discussions and references about the speed of convergence in CLT for mixing sequences.

\section{Odometer and sum-of-digits function}\label{Odometer and sum-of-digits function}
    \subsection{Unique ergodicity of the \texorpdfstring{$b$}{b}-adic odometer}
We define $\XX$ as the space of \emph{$b-$adic integers}, that is the space $\{0,\cdots,b-1\}^{\NN}$. Coordinates of a $b$-adic integer $x\in \XX$ are interpreted as digits in base  $b$: elements of $\XX$ can be viewed as ``generalized integers having possibly infinitely many non zero digits in base $b$''. To comply with the usual writing of numbers in base $10$, an element $x=(x_n)_{n\in\NN}\in \XX$ will be represented as a left-infinite sequence $(\cdots, x_1,x_0)$, $x_0$ being the units digit. The space $\XX$ is compact for the product topology. The set of integers $\NN$ can be identified with the subset of sequences with finite support. More precisely, using the inclusion function 
\begin{align*}
    i:\overline{n_{\ell}\cdots n_0} \in\NN \longmapsto (\cdots, 0,n_{\ell},\cdots, n_1,n_0) \in \XX
\end{align*}
we identify $\NN$ and $i(\NN)$. 

$\XX$ is equipped with an addition which extends the usual addition on $\NN$ and turns $\XX$ into an Abelian group. For $x=(\cdots,x_0)$ and $y=(\cdots,y_0)$ in $\XX$, $(x+y)$ is determined recursively by the following process, in which we generate a sequence of carries $(c_{\ell})_{\ell\ge 0}\in\{0,1\}^{\NN}$.
    \begin{itemize}
        \item Initialisation: \\        
            \hspace{5mm} if $x_0+y_0<b$ then we set $(x+y)_0:=x_0+y_0$ and $c_0:=0$, \\
            \hspace{5mm} else we set $(x+y)_0:=x_0+y_0-b$ and $c_0:=1$.
        \item Induction step: once, for some $\ell\ge 1$, we have computed  $(x+y)_i$ and $c_i$ for $i=0, \cdots, \ell-1$, \\
            \hspace{5mm}  if $x_{\ell}+y_{\ell}+c_{\ell-1}<b$ then we set $(x+y)_{\ell}:=x_{\ell}+y_{\ell}+c_{\ell-1}$  and $c_{\ell}:=0$, \\
                 \hspace{5mm} else we set $(x+y)_{\ell}:=x_{\ell}+y_{\ell}+c_{\ell-1}-b$ and $c_{\ell}:=1$.
    \end{itemize}
Now, since $1$ belongs to $\NN\subset \XX$, we can consider the application $T:\XX\rightarrow \XX$
\begin{align*}
    T:x\longmapsto T(x):=x+1,
\end{align*}
which is usually refered to as the \emph{$b$-adic odometer}.
It is well-known that $T$ is a homeomorphism  and so $(\XX,T)$ is a topological dynamical system $\cite{NPF}$. For $\ell\ge 0$ and for integers $r_{\ell},\cdots, r_0\in\{0,\cdots, b-1\}$, we define the \emph{cylinder} $C_{r_{\ell}\cdots r_0}$ as the set of sequences $x$ such that $x_i=r_i$ for $ i=0,\cdots, \ell$. We observe that the image by $T$ of a cylinder is another cylinder~: for integers $r_{\ell},\cdots, r_0\in\{0,\cdots, b-1\}$, if there exists a minimal index $i\in\{0,\cdots,\ell\}$ such that $r_i\neq b-1$ then $TC_{r_{\ell}\cdots r_0}=C_{r_{\ell}\cdots r_{i+1}(1+r_i)0^i}$, otherwise $TC_{r_{\ell}\cdots r_0}=C_{0^{\ell+1}}$. Also, we define the \emph{Rokhlin tower of order $\ell\ge 0$} as the family  
\begin{align*}
    \left(C_{0^{\ell+1}}, TC_{0^{\ell+1}}, \cdots, T^{b^{\ell+1}-1}C_{0^{\ell+1}}\right)
\end{align*}
where $(T^jC_{0^{\ell+1}})_{0\le j\le b^{\ell+1}-1}$ form a partition of $\XX$. We commonly represent this family as a tower as shown in Figure \ref{rok1}.
    \begin{figure}[H]
        \centering
        \includegraphics[scale=0.28]{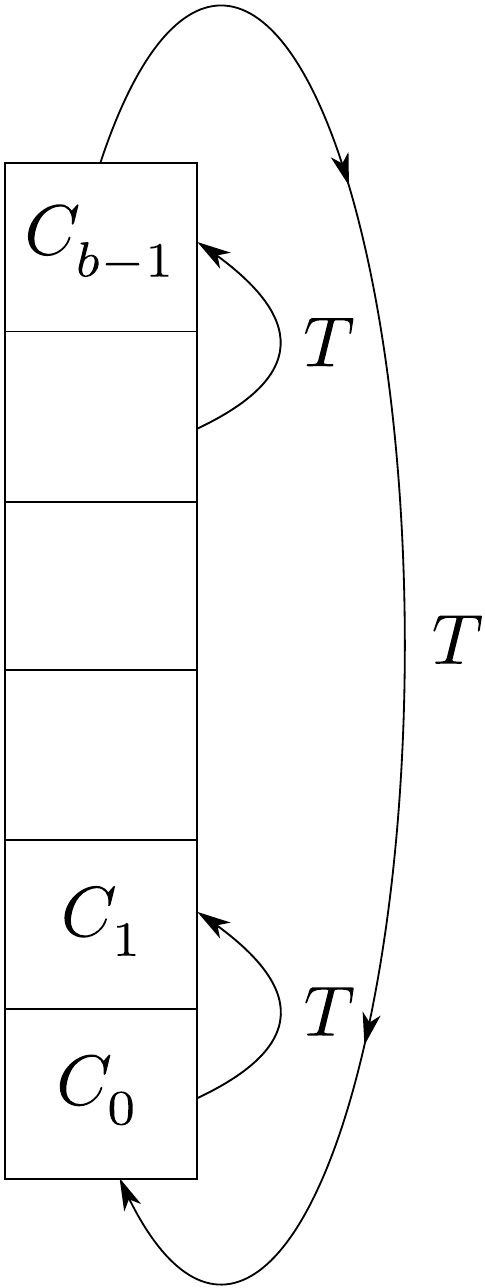}
        \caption{Behavior of $T$ on the Rokhlin tower of order $0$.}
        \label{rok1}
    \end{figure}
It is classical (see the survey \cite{TD}) that the sequence of towers can be constructed with the so-called \emph{Cut-and-Stack} inductive process, as illustrated in Figure \ref{rok2}.

\begin{figure}[H]
    \centering
    \includegraphics[scale=0.34]{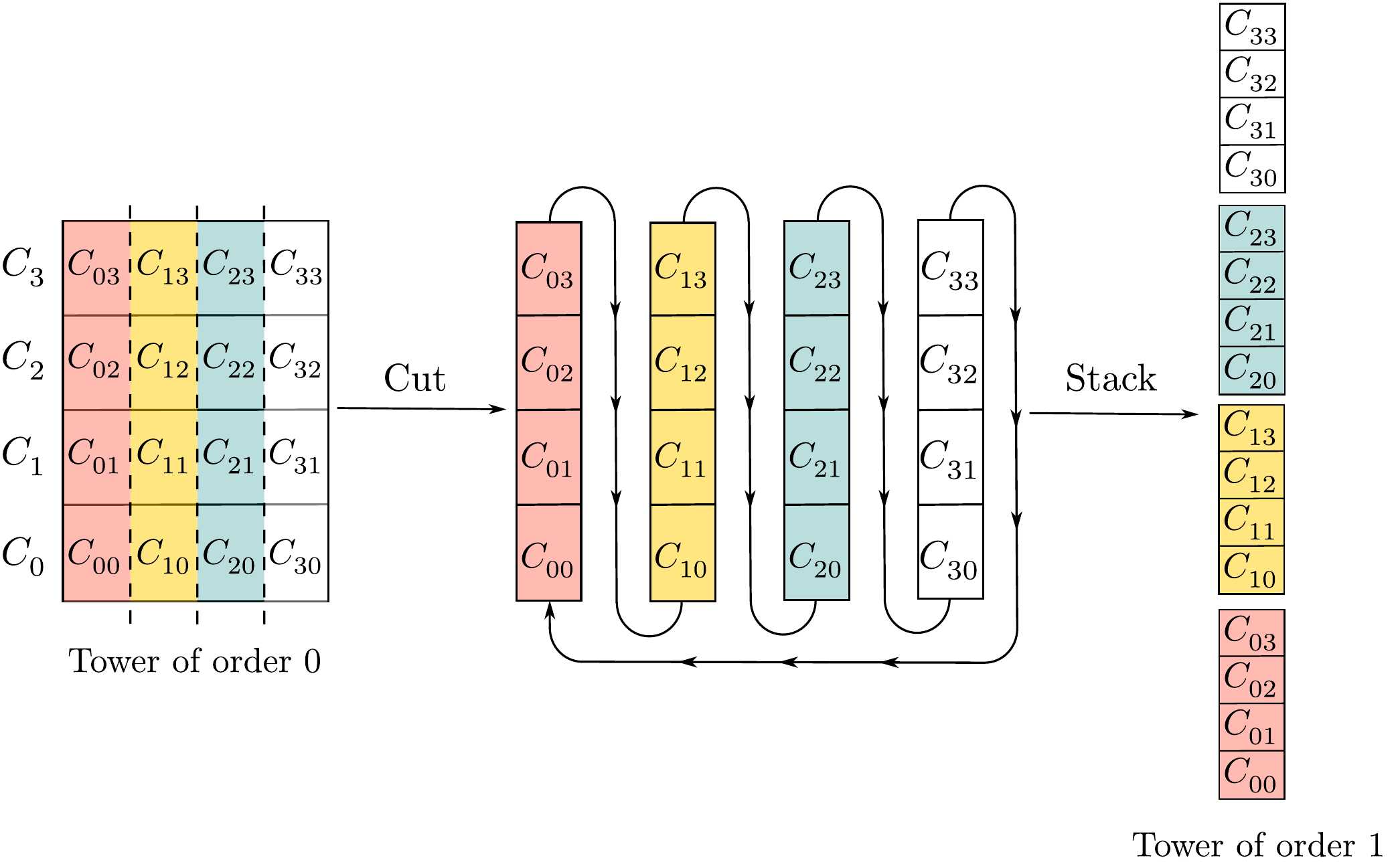}
    \caption{How to construct the tower of order $1$ of $T$ from the tower of order $0$ in base $b=4$.}
    \label{rok2}
\end{figure}

By looking at the behavior of $T$ on these towers, one can show that if $\PP$ is a $T-$invariant probability measure on $(\XX,T)$, then $\PP$ gives the same measure to each level in a given tower: for every $\ell\ge 0$ and $r_0,\cdots,r_{\ell}\in\{0,\cdots, b-1\}$
\begin{align*}
    \PP\left(C_{r_{\ell}...r_0}\right)=\frac{1}{b^{\ell+1}}.
\end{align*}
Since the cylinders generate the Borel $\sigma$-algebra, $\PP$ is uniquely determined by these values on the cylinders, hence $(\XX,T)$ is a uniquely ergodic dynamical system.
We observe that choosing $x$ in $\XX$ according to the unique $T$-invariant law $\PP$ means choosing its digits independently according to the uniform law on $\{0,\cdots, b-1\}$. We also note that $\PP$ is the normalized Haar measure on $\XX$.

For $x$ in $\XX$, we define the sequence of empirical probability measures along the (beginning of the) orbit of $x$: for every $N\ge 1$, we set
\[\epsilon_N(x):=\frac{1}{N}\sum_{0\le n<N}\delta_{T^nx}\] 
(where $\delta_y$ denotes the Dirac measure on $y\in\XX$).\\
Since the space of probability measures on $\XX$ is compact for the weak-$*$ topology, every subsequential limit of $(\epsilon_N(x))$ is a $T$-invariant probability measure. By the uniqueness of the $T$-invariant probability measure, for every $x\in \XX$ we have $\epsilon_N(x)\rightarrow \PP$. In other words, we have the convergence 
\begin{equation}\label{uniq_ergo+birk1}
        \forall x\in \XX, \hspace{2mm} \forall f\in\mathcal{C}(\XX), \hspace{5mm} \frac{1}{N}\sum_{0\le n< N}f(T^nx) \tend{N}{+\infty}{}{\int_{\XX} f \mathrm{d}\PP}.
\end{equation}
We will be interested here in the special case $x=0$ because $\NN=\{T^n0:n\in\NN\}$. Then \eqref{uniq_ergo+birk1} becomes
\begin{equation}\label{uniq_ergo+birk}
        \forall f\in\mathcal{C}(\XX), \hspace{5mm} \frac{1}{N}\sum_{0\le n< N}f(n) \tend{N}{+\infty}{}{\int_{\XX} f \mathrm{d}\PP}.
\end{equation}
Equation \eqref{uniq_ergo+birk} shows that, for a continuous function $f$, averaging $f$ over $\NN$ (for the natural density) amounts to averaging over $\XX$ (for $\PP$). The next section shows how this convergence can be extended to some non-continuous functions related to the sum-of-digits function.

    \subsection{Sum of digits on the odometer}\label{Sum of digits on the odometer}
For every integer $k$,  we define $s_k:\XX\rightarrow\ZZ$ as the sum of the first $(k+1)$ digits function, that is to say 
\begin{align*}
    s_k(x):=x_0+\cdots +x_k.    
\end{align*}
Let $r\in\NN$. We define the functions $\Delta^{(r)}_k:\XX\rightarrow\ZZ$ by 
\begin{align*}
    \Delta_k^{(r)}(x):=s_k(x+r)-s_k(x).
\end{align*}
The functions $\Delta_k^{(r)}$ are well-defined, continuous (and bounded) on $\XX$. By \eqref{uniq_ergo+birk}, we have 
\begin{align}\label{birkhoff_result_delta_k}
         \frac{1}{N}\sum_{n< N} \Delta_k^{(r)}(n)&=\frac{1}{N}\sum_{n< N} \Delta_k^{(r)}(T^n0)\tend{N}{+\infty}{}{\int_{\XX} \Delta_k^{(r)} \mathrm{d}\PP}.
\end{align}
Although the sum-of-digits function $s$ is not well defined on $\XX$, we can extend the function $\Delta^{(r)}$ defined in \eqref{refDeltar} on the set of $x\in\XX$ for which the number of different digits between $x$ and $x+r$ is finite. This subset contains the $b$-adic integers $x$ such that there exists an index $k\ge \max(\{\ell: r_{\ell}\neq 0\}$ such that $x_k\neq b-1$. So, except for a finite number of $b$-adic integers, we can define 
\[\Delta^{(r)}(x):=\lim\limits_{k\to\infty}\Delta^{(r)}_k(x).\]

\begin{lemme}
    Let $t\ge 1$. We have $\PP$-almost surely the following identity 
    \begin{align}\label{decomp_gene}
    \Delta^{(t)}&=\Delta^{(1)}+\Delta^{(1)}\circ T+\cdots+\Delta^{(1)}\circ T^{t-1}.
\end{align}
\end{lemme}
\begin{proof}
    For every integers $k$ and $u$, we have the decomposition formula 
    \begin{align}\label{decomp_sumk}
        \Delta_k^{(t+u)}&=\Delta_k^{(t)}+\Delta_k^{(u)}\circ T^t.
    \end{align}
    So, taking $\PP$-almost everywhere the limit when $k$ tends to infinity, we get
    \begin{align}\label{decomp_sum}
        \Delta^{(t+u)}&=\Delta^{(t)}+\Delta^{(u)}\circ T^t \hspace{2mm}  (\PP\mbox{-a-s.}).
    \end{align}
    An induction on $t$ gives \eqref{decomp_gene}.
\end{proof}

We observe that $\Delta^{(r)}$ also satisfies \eqref{lien_delta_retenue_entier}: for each $x\in \XX$ for which $\Delta^{(r)}(x)$ is well defined, we have
\begin{align}\label{lien_delta_retenue_odometre}
    \Delta^{(r)}(x)=s(r)-c (b-1),
\end{align}
where $c:=\sum_{\ell\ge 0} c_{\ell}<\infty$ is the total number of carries generated during the computation of $x+r$.\\
Unfortunately, $\Delta^{(r)}$ is not continuous like the functions $\Delta_k^{(r)}$, it is not even bounded on $\XX$, but we have the following result about functions with polynomial growth.

\begin{prop}\label{prop_moment}
Let $r\ge 1$ and $f:\ZZ\rightarrow\CC$. Assume that there exist  $\alpha\ge 1$ and $C$ in $\RR^*_+$  such that for every $n\in\ZZ$  
\begin{align}\label{hyp_prop_moment2}
    \lvert f(n) \rvert &\le C\lvert n \rvert^{\alpha} +\lvert f(0)\rvert.
\end{align}
Then $f\circ \Delta^{(r)}\in L^1\left(\PP\right)$ and we have the convergence
\begin{align*}
    \lim\limits_{N\to\infty}\frac{1}{N}\sum_{n<N}f(\Delta^{(r)}(n))&=\int_{\XX} f(\Delta^{(r)}(x)) \mathrm{d}\PP(x)\\
            &= \lim\limits_{k\to\infty} \int_{\XX} f(\Delta^{(r)}_k(x)) \mathrm{d}\PP(x).
\end{align*}
\end{prop}

\begin{corollaire}\label{corollaire_prop_moment}
 For every $d\in\ZZ$
\begin{align}\label{lien_mu^r_et_mu}
        \mu^{(r)}(d)&:=\lim\limits_{N\to\infty}\frac{1}{N} \left|\left\{n<N: \Delta^{(r)}(n)=d\right\}\right| \notag \\
        &=\PP\left(\left\{x\in \XX: \Delta^{(r)}(x)=d\right\}\right).
\end{align}
Moreover, $\Delta^{(r)}$ has zero-mean and has finite moments. 
\end{corollaire}
In particular, we recover Bésineau's result on the existence of the asymptotic density and the fact that $\sum_{d\in\ZZ} \mu^{(r)}(d)=1$.
\begin{remarque}\label{et pour 0 ?}
    Using trivial arguments, Proposition~\ref{prop_moment} and Corollary~\ref{corollaire_prop_moment} are also true when $r=0$. We observe that $\mu^{(0)}=\delta_0$.
\end{remarque}

Before proving this proposition and its corollary, we need some lemmas.
\begin{lemme}\label{lemme_technique}
    Let $r\ge 1$. For $N\in\NN^{*}$, for $k\in\NN$ and $d,d'\in\ZZ$, we have the inequality
    \begin{equation}\label{ineg_techni2}
        \begin{split}
             \frac{1}{N}\left| \{n<N: (\Delta^{(r)}(n),\Delta^{(r)}_k(n))=(d,d')\} \right| & \\
             & \hspace{-3 cm}\le rb  \hspace{1mm}\PP \hspace{-1mm} \left( \{x\in \XX: (\Delta^{(r)}(x),\Delta^{(r)}_k(x))=(d,d')\}\right).
        \end{split}
    \end{equation}
    In particular, we have
    \begin{align}\label{ineg_techni}
        \frac{1}{N}\left| \{n<N: \Delta^{(r)}(n)=d\} \right| \le rb \hspace{1mm} \PP \hspace{-1mm} \left( \{x\in \XX: \Delta^{(r)}(x)=d\}\right).
    \end{align}

\end{lemme}
\begin{proof}
Of course, \eqref{ineg_techni2} implies \eqref{ineg_techni} so we just need to prove \eqref{ineg_techni2}. We fix $k\in\NN$. For every $\ell\in\NN$, let $V_{\ell}$ be the set of the values reached by the couple $(\Delta^{(r)},\Delta^{(r)}_k)$ on the first $b^{\ell+1}-r$ levels of the Rokhlin tower of order $\ell$ (see Figure \ref{descr_Vl}). Of course, if $b^{\ell+1}-r\le 0$  then $V_{\ell}:=\emptyset$. Otherwise, we observe that $V_{\ell}$ is a finite set. Indeed, the first $b^{\ell+1}-r$ levels correspond to the $b$-adic integers $x$ such that, when we add $r$, the carry propagation does not go beyond the first $\ell+1$ digits. Since these digits are fixed on a level of the Rokhlin tower of order $\ell$, except for the last $r$ levels, $\Delta^{(r)}$ and $\Delta^{(r)}_k$ are constant on each such level. We observe that the sequence $(V_{\ell})_{\ell\ge 0}$ is increasing for the inclusion.

Now, for $d,d'\in\ZZ$, there are $2 $ cases.
\begin{enumerate}
    \item If $(d,d')\notin\cup_{\ell\ge 0} V_{\ell}$, then for each $n\in\NN$ we have $(\Delta^{(r)}(n),\Delta^{(r)}_k(n))\neq(d,d')$. Indeed, for each $n\in\NN$, there exists a smallest integer $\ell$ such that $n$ is in the first $b^{\ell+1}-r$ levels of the tower of order $\ell$, hence $(\Delta^{(r)}(n),\Delta^{(r)}_k(n)) \\ \in V_{\ell}$. In this case, \eqref{ineg_techni2} is trivial.
    \item If $(d,d')\in\cup_{\ell\ge 0} V_{\ell}$ then there exists a unique $\ell\ge 0$ such that $(d,d')\in V_{\ell}\backslash V_{\ell-1}$ (with the convention $V_{-1}:=\emptyset$). Since $(\Delta^{(r)},\Delta^{(r)}_k)$ is constant on each of the first $b^{\ell+1}-r$ levels of the tower, it takes the value $(d,d')$ on at least one whole such level of measure $\frac{1}{b^{\ell+1}}$. So, we have  
    \[\PP\hspace{-1mm} \left( \{x\in \XX: (\Delta^{(r)}(x),\Delta^{(r)}_k(x))=(d,d')\}\right)\ge \frac{1}{b^{\ell+1}}.\]
    Also, since the couple $(d,d')$ does not appear in the first levels of the previous tower ($(d,d')\notin V_{\ell-1}$), we claim that, for every $N\ge 1$
    \[\frac{1}{N}\left| \{n<N: (\Delta^{(r)}(n),\Delta^{(r)}_k(n))=(d,d')\} \right|\le \frac{r}{b^{\ell}}.\]
    Indeed,
    \begin{enumerate}
    \item If $r\ge b^{\ell}$, then the inequality is then trivial.
    \item If $r< b^{\ell}$ then, since $(d,d')$ is not in $V_{\ell-1}$, $(d,d')$ can only appear inside the  $r$ highest levels of the tower of order $\ell-1$. So, if we note $C$ the union of these $r$ highest levels, using the fact that $0$ is in the first level of the tower, we have
    \begin{align*}
        \frac{1}{N}\left| \{0\le n<N: (\Delta^{(r)}(n),\Delta^{(r)}_k(n))=(d,d')\} \right|&\\
        &\hspace{-1.38cm}\le \frac{1}{N}\left| \{0\le n<N: T^n0\in C \}\right|\\
        &\hspace{-1.38cm}\le\frac{r}{b^{\ell}}.
    \end{align*}
\end{enumerate} 
    Combining both inequalities gives \eqref{ineg_techni2}.
\end{enumerate} 

\begin{figure}[H]
    \centering
    \includegraphics[scale=0.35]{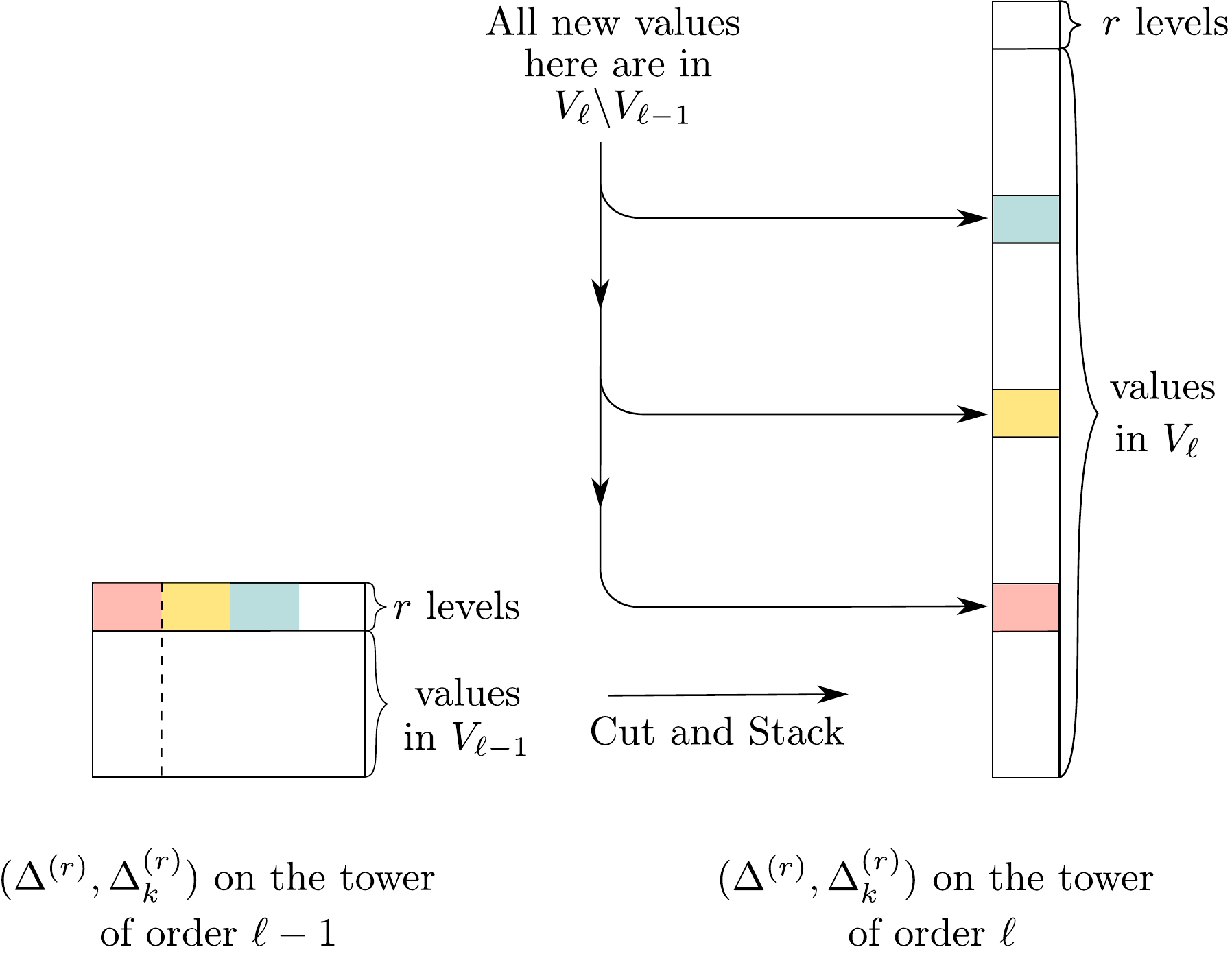}
    \caption{Visual description of $V_{\ell-1}$, $V_{\ell}$ and $V_{\ell}\setminus V_{\ell-1}$.}
    \label{descr_Vl}
\end{figure}
\end{proof}

We also need the following lemma.
\begin{lemme}\label{g_i is in L1}
Let $i\in\NN$. For every $x\in\XX$, we define the function
\[g_i(x):=\sup_{k\in\NN}\lvert\Delta_k^{(1)}\circ T^i(x)\rvert.\]
Then, for any $N\in\NN$, $g_i^N\in L^1(\PP)$.
\end{lemme}
\begin{proof}
It is equivalent to show that $\sum_m \PP\left(\{x\in \XX: g_i(x)^{N}>m\}\right)$ is a convergent series. We have 
\begin{align*}
    g_i(x)> m^{\frac{1}{N}} &\Leftrightarrow \sup_{k\in\NN}|\Delta^{(1)}_k\circ T^i(x)|> m^{\frac{1}{N}}\\&\Leftrightarrow \exists k\in\NN, \hspace{2mm} \lvert\Delta_k^{(1)}(T^ix) \rvert > m^{\frac{1}{N}}. \\
\end{align*}
From \eqref{lien_delta_retenue_odometre}, this condition is true if and only if  the addition $T^ix+1$ creates sufficiently many carries: strictly more than $\lfloor \frac{m^\frac{1}{N}+1}{b-1}\rfloor$. But, we know that when we add $1$ to $T^ix$, the number of carries created by the addition is the number of $(b-1)$'s at the right-hand side of the expansion of $T^i x$. So, because of the $T$-invariance of $\PP$
 \begin{align*}
     \PP\left(\{x\in \XX: g_i(x)>m^{\frac{1}{N}}\}\right) &\le \left(\frac{1}{b}\right)^{\left\lfloor \frac{m^\frac{1}{N}+1}{b-1}\right\rfloor}.
 \end{align*}
 The right quantity is the general term of a convergent series which shows that $g_i^{N}\in L^1(\PP)$.
\end{proof}

\begin{proof}[Proof of Proposition \ref{prop_moment}]
Let $\varepsilon>0$. For any integer $k$, we have
\begin{align*}
    \left\vert\frac{1}{N}\sum_{n<N} f(\Delta^{(r)}(n))-\int_{\XX} f(\Delta^{(r)}(x)) \mathrm{d}\PP(x)\right\vert 
        &\le A_1+A_2+A_3,
\end{align*}
where
\begin{align*}
    A_1 &:= \displaystyle\frac{1}{N}\sum_{n<N} \left\vert f(\Delta^{(r)}(n))-f(\Delta^{(r)}_k(n))\right\vert, \\ 
    A_2 &:= \displaystyle\left\vert\frac{1}{N}\sum_{n<N}f(\Delta^{(r)}_k(n))-\int_{\XX} f(\Delta^{(r)}_k(x)) \mathrm{d}\PP(x)\right\vert, \\
    A_3 &:=\displaystyle\int_{\XX} \left\vert f(\Delta^{(r)}(x))-f(\Delta^{(r)}_k(x))\right\vert \mathrm{d}\PP(x).
\end{align*}

For $A_1$, using Lemma~\ref{lemme_technique}, we get
\begin{align*}
    A_1&=\frac{1}{N}\sum_{n<N} \sum_{j,j'\in\ZZ}\Big\vert f(j)-f(j')\Big\vert\mathbbm{1}_{(j,j')}\left(\Delta^{(r)}(n),\Delta^{(r)}_k(n)\right)\\
    &= \sum_{j,j'\in\ZZ}\Big\vert f(j)-f(j')\Big\vert\frac{1}{N}\sum_{n<N}\mathbbm{1}_{(j,j')}\left(\Delta^{(r)}(n),\Delta^{(r)}_k(n)\right)\\
    &\le rb\sum_{j,j'\in\ZZ}\Big\vert f(j)-f(j')\Big\vert \hspace{1mm}\PP \hspace{-1mm} \left( \Delta^{(r)}(n) = j , \Delta^{(r)}_k(n)=j'\right)\\
    &=rb\int_{\XX} \Big\vert f(\Delta^{(r)}(x))-f(\Delta^{(r)}_k(x))\Big\vert\mathrm{d}\PP(x)=rb A_3.
\end{align*}
Hence, controlling $A_3$ also enables us to control $A_1$. For this, we note that the integrand in $A_3$ converges $\PP$-almost everywhere to $0$, therefore it is enough to show that the dominated convergence theorem applies. To find a good dominant function, we write using \eqref{hyp_prop_moment2}
\begin{align*}
    \left\lvert f\circ\Delta^{(r)}_k(x)\right\rvert&\le C\left\lvert\Delta^{(r)}_k(x)\right\rvert^{\alpha}+\Big\lvert f(0)\Big\rvert.\\
\end{align*}
We get from \eqref{decomp_sumk}
 \[\Delta^{(r)}_k=\Delta^{(1)}_k +\Delta^{(1)}_k\circ T + \cdots + \Delta^{(1)}_k\circ T^{r-1},\]
then we use the multinomial theorem (we can suppose $\alpha\in\NN)$ to write
\begin{align*}
    \left\lvert\Delta^{(1)}_k(x)+\cdots+\Delta^{(1)}_k\circ T^{r-1}(x)\right\rvert^{\alpha}
    & = \sum_{j_0+...+j_{r-1}=\alpha} \binom{\alpha}{j_0,...,j_{r-1}} \prod_{i=0}^{r-1} \left\lvert \Delta^{(1)}_k \circ T^i(x) \right\rvert^{j_i}.\\
\end{align*}
We now use Young's inequality to get
\begin{align*}
    \prod_{i=0}^{r-1} \left\lvert \Delta^{(1)}_k \circ T^i(x) \right\rvert^{j_i}&\le \frac{1}{r}\sum_{i=0}^{r-1}\left\lvert\Delta^{(1)}_k\circ T^i(x)\right\rvert^{rj_i}.\\
\end{align*}
Then, using the function $g_i$ defined in Lemma~\ref{g_i is in L1}, we get the inequality
\[\left\lvert f\circ\Delta^{(r)}_k(x)\right\rvert\le C \sum_{j_0+\cdots+j_{r-1}=\alpha}\sum_{i=0}^{r-1}\frac{\binom{\alpha}{j_0,\cdots,j_{r-1}}}{r}g_i(x)^{rj_i}+\Big\lvert f(0)\Big\rvert.\]
Moreover, when it is well defined, we have $\Delta^{(r)}\circ T^i=\lim\limits_{k\to\infty}\Delta^{(r)}_k\circ T^i$ therefore, we also get $\left\vert\Delta^{(r)}\circ T^i\right\vert\le g_i$ which yields the similar inequality
\[\left\lvert f\circ\Delta^{(r)}(x)\right\rvert\le C \sum_{j_0+\cdots+j_{r-1}=\alpha}\sum_{i=0}^{r-1}\frac{\binom{\alpha}{j_0,\cdots,j_{r-1}}}{r}g_i(x)^{rj_i}+\Big\lvert f(0)\Big\rvert.\]
As proved in Lemma~\ref{g_i is in L1}, $g_i^{rj_i}\in L^1(\PP)$ thus the dominated convergence theorem can be applied and, for $k$ large enough, $A_1+A_3\le \frac{\varepsilon}{2}$ for every $N\ge 1$.\\
 Now, once we have fixed such a $k$, for $N$ large enough, $A_2$ is bounded by $\frac{\varepsilon}{2}$ because of \eqref{uniq_ergo+birk} and the continuity of $\Delta_k^{(r)}$ and of $f$. The convergence in the statement is thus proved.
 
 Note that the argument of the dominated convergence theorem also proves that $f\circ\Delta^{(r)}\in L^1(\PP)$ and $\int_{\XX} f\circ \Delta^{(r)}\mathrm{d}\PP=\lim\limits_{k\rightarrow \infty}\int_{\XX} f\circ\Delta^{(r)}_k\mathrm{d}\PP$.
 \end{proof}
 \begin{proof}[Proof of Corollary \ref{corollaire_prop_moment}]
 We just apply Proposition \ref{prop_moment} with particular functions $f$. First, for $d\in\ZZ$, we use the function $f=\mathbbm{1}_{\{d\}}$. It gives
 \[\mu^{(r)}(d)=\PP\left(\left\{x\in \XX: \Delta^{(r)}(x)=d\right\}\right).\]
 Then, we take $f$ as the identity function on $\ZZ$ for which $\eqref{hyp_prop_moment2}$ is clearly satisfied. Also, for every integer $k$, we have by $T$-invariance of $\PP$
    \[\int_{\XX} \Delta^{(r)}_k \mathrm{d}\PP=\int_{\XX} s_k\circ T^r\mathrm{d}\PP - \int_{\XX} s_k\mathrm{d}\PP=0.\]
We then deduce that
\begin{align}\label{esperance_nulle}
    \sum_{d\in\ZZ} d\,\mu^{(r)}(d) = \int_{\XX} \Delta^{(r)} \mathrm{d}\PP=\lim\limits_{k\rightarrow \infty}\int_{\XX} \Delta^{(r)}_k \mathrm{d}\PP=0.  
\end{align}
Finally, we use, for every $j\ge 2$, the function $f(n)=n^j$ which satisfies \eqref{hyp_prop_moment2}. This gives the existence of moments of order $j$ for $\Delta^{(r)}$.
\end{proof}

More generally, we have the following convergence.
\begin{prop}\label{prop_gene}
Let $r\ge 1$ and $f:\ZZ\rightarrow\CC$ be such that $f\circ\Delta^{(r)}\in L^1(\PP)$. Then 
    \[\lim\limits_{N\to\infty}\frac{1}{N}\sum_{n<N}f(\Delta^{(r)}(n))=\int_{\XX} f(\Delta^{(r)}(x)) \mathrm{d}\PP(x).\]
\end{prop}

\begin{proof}[Proof of Proposition~\ref{prop_gene}.]
By \eqref{lien_delta_retenue_odometre}, the values reached by $\Delta^{(r)}$ are of the form $a_k^{(r)}:=s(r)-k(b-1)$, $k\ge 0$, and we have
    \begin{align}\label{2}
        \int_{\XX} f\left(\Delta^{(r)}(x)\right) \mathrm{d}\PP(x)&=\sum_{k\ge 0} f\left(a_k^{(r)}\right)\PP\left(\{x\in \XX: \Delta^{(r)}(x)=a_k^{(r)}\}\right).
    \end{align}
On the other hand, we write
    \begin{align*}
        \frac{1}{N}\sum_{n<N}f\left(\Delta^{(r)}(n)\right)  &=\frac{1}{N}\sum_{n<N}f\left(\Delta^{(r)}(n)\right)\sum_{k\ge 0}\mathbbm{1}_{\left\{a_k^{(r)}\right\}}\left(\Delta^{(r)}(n)\right) \notag\\
        &=\sum_{k\ge0}\underbrace{f\left(a_k^{(r)}\right)\frac{1}{N}\sum_{n<N}\mathbbm{1}_{\left\{a_k^{(r)}\right\}}\left(\Delta^{(r)}(n)\right)}_{:=u_{N,k}}.
    \end{align*}
We conclude by applying the dominated convergence theorem in the metric space $\ell^1(\NN)$ endowed with the counting measure to show 
\[\sum_{k\ge0} u_{N,k}\tend{N}{+\infty}{}{\sum_{k\ge 0} f\left(a_k^{(r)}\right)\PP\left(\{x\in \XX: \Delta^{(r)}(x)=a_k^{(r)}\}\right)}.\]
\begin{enumerate}
    \item The pointwise limit on $\NN$ of $u_{N,k}$ is $f\left(a_k^{(r)}\right)\PP\left(\{x\in \XX: \Delta^{(r)}(x)=a_k^{(r)}\}\right)$ by \eqref{lien_mu^r_et_mu}.
    \item A dominant function is given using Lemma \ref{lemme_technique}: for all $k$ and $N$ we have
    \[u_{N,k}\le g(k):=\left\lvert f\left(a_k^{(r)}\right)\right\rvert rb \hspace{1mm} \PP \hspace{-1mm} \left( \{x\in \XX: \Delta^{(r)}(x)=a_k^{(r)}\}\right).\]
    By \eqref{2} and the hypothesis that $f\circ\Delta^{(r)}$ is integrable, $\sum g(k)$ is convergent. 
\end{enumerate}
\end{proof}

We have shown that, for any integer $r$, $\Delta^{(r)}$ has finite moments. The next section will focus on the second moment, which is the one we are most interested in for our CLT.

\section{Variance of \texorpdfstring{$\mu^{(r)}$}{Dr}}\label{Variance of Delta}
        \subsection{Inductive relation on the measures}
Given $r\in\NN$, there exist $\widetilde{r}\in\NN$ and $0\le r_0\le b-1$ such that $r=b\widetilde{r}+r_0$. The integer $r_0$ is actually the units digit of the expansion of $r$ and, if $r\ge b$, $\widetilde{r}$ corresponds to the integer whose expansion is obtained by erasing $r_0$ in the expansion of $r$. The expansion of $\widetilde{r}$ is then one digit shorter than the expansion of $r$. If $r<b$ then, of course, $r_0=r$ and $\widetilde{r}=0$. 
First of all, we have a well known inductive relation on the length of the expansion of $r$.
\begin{prop}\label{relat_rec_mesure}
For $\widetilde{r}\in\NN$, $0\le r_0\le b-1$ and $d\in\ZZ$ 
\begin{align}\label{relat rec mu}
    \mu^{(b\widetilde{r}+r_0)}(d)&=\frac{b-r_0}{b}\mu^{(\widetilde{r})}(d-r_0)+\frac{r_0}{b}\mu^{(\widetilde{r}+1)}(d+b-r_0).
\end{align}
\end{prop}
\begin{proof}
Let $x=(\cdots,x_1, x_0)\in \XX$. We define $\widetilde{x}:=(\cdots,x_1)$. Let us consider the computation of the digits of $x+r$, where $r:=b\widetilde{r}+r_0=\overline{r_{\ell}\cdots r_0}$

\[\begin{array}{cccccc}
        & \cdots    & x_{\ell}  & \cdots    & x_1   & x_0 \\
    +   &           & r_{\ell}  & \cdots    & r_1   & r_0 \\ \hline
    =   &           &           & \cdots    &       & (x+r)_0
\end{array}\]

If $x_0+r_0<b$, no carry is created and $\Delta^{(b\widetilde{r}+r_0)}(x)=r_0+\Delta^{(\widetilde{r})}(\widetilde{x})$. Otherwise,  $x_0+r_0\ge b$ and we have to subtract $b$ from the units digit of the result: we are left with the addition of $\widetilde{x}$ and $\widetilde{r}+1$: so $\Delta^{(b\widetilde{r}+r_0)}(x)=r_0-b+\Delta^{(\widetilde{r}+1)}(\widetilde{x})$.
To sum up, we have 
\[\Delta^{(b\widetilde{r}+r_0)}(x)=\left\{\begin{array}{ll}
     r_0+\Delta^{(\widetilde{r})}(\widetilde{x}) & \mbox{ if }x_0+r_0<b,  \\
     r_0-b+\Delta^{(\widetilde{r}+1)}(\widetilde{x}) & \mbox{ otherwise.}  \\ 
\end{array}\right.\]
Now, let $d\in\ZZ$. We partition the set $\{x\in \XX: \Delta^{(b\widetilde{r}+r_0)}(x)=d\}$ according to the value of $x_0$
\begin{align*}
    \{x\in \XX: \Delta^{(b\widetilde{r}+r_0)}(x)=d\}&=\bigcup_{j=0}^{b-r_0-1} \{x\in \XX: x_0=j \mbox{ and }\Delta^{(\widetilde{r})}(\widetilde{x})=d-r_0\} \\
    & \hspace{-1.7mm} \bigcup \bigcup_{j=b-r_0}^{b-1}\{x\in \XX: x_0=j \mbox{ and }\Delta^{(\widetilde{r}+1)}(\widetilde{x})=d+b-r_0\}.
\end{align*}
We observe that if $x$ is randomly chosen with law $\PP$, then $\widetilde{x}$ is independent of $x_0$ and also follows $\PP$. We just need to take the measure to conclude. 
\end{proof}

If we apply finitely many times \eqref{relat rec mu}, we can express $\mu^{(r)}(d)$ as a convex combination of the measures $\mu^{(0)}$ and $\mu^{(1)}$ evaluated on particular points. We recall that $\mu^{(0)}=\delta_0$ (Dirac measure on $0$).
We can also compute $\mu^{(1)}$.
\begin{lemme}
For every $d\in\ZZ$
        \[\mu^{(1)}(d):=\left\{\begin{array}{ll}
         \frac{1}{b^{k}}-\frac{1}{b^{k+1}} & \mbox{if }d=1-k(b-1) \mbox{ for some }k\in\NN  \\
         0& \mbox{otherwise.}
    \end{array}\right.
    \]
\end{lemme}

\begin{proof}
We use again the notation $a_k^{(1)}=1-k(b-1).\\$
By \eqref{lien_delta_retenue_odometre}, it is trivial that if $d\neq a_k^{(1)}$ for all $k\in\NN$ then $\mu^{(1)}(d)=0$. Otherwise, we recall again by \eqref{lien_delta_retenue_odometre} that for all $k\in\NN$, $\Delta^{(1)}(x)=a_k^{(1)}$ if and only if the right-hand side of the expansion of $x$ is exactly a block of $(b-1)$'s of length $k$. Thus, if $k\ge 1$
\begin{align*}
    \mu^{(1)}\left(a_k^{(1)}\right)&=\PP \Big( \{x\in \XX: x_0=\cdots=x_{k-1}=b-1 \mbox{ and } x_k<b-1\} \Big)\\
                        &=\frac{b-1}{b^{k+1}}.
\end{align*}    
And, if $k=0$
    \[\mu^{(1)}\left(a_k^{(1)}\right)=\PP \Big( \{x\in \XX: x_0<b-1\} \Big)=\frac{b-1}{b}.\]
\end{proof}
\begin{figure}[H]
    \centering
    \includegraphics[scale=0.4]{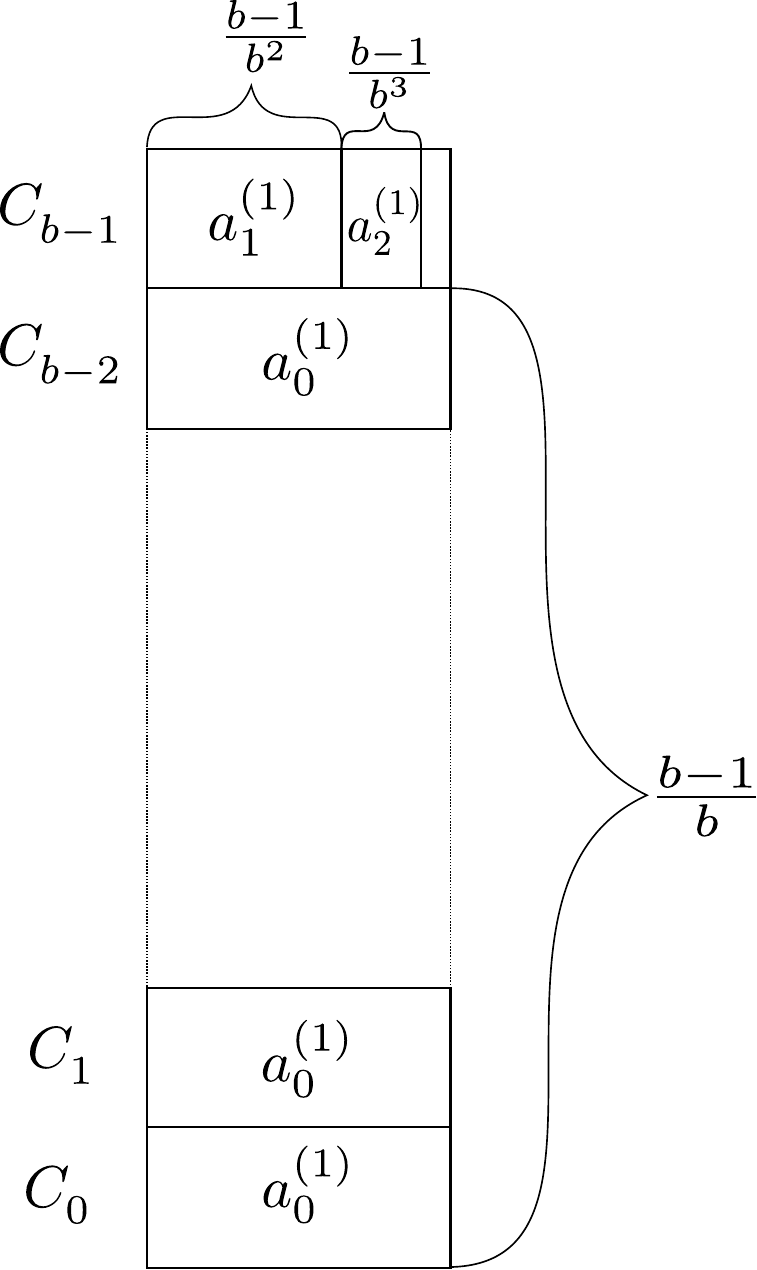}
    \caption{Values of $\Delta^{(1)}$ on the levels of the Rokhlin tower of order $0$, and the corresponding $\PP$-measures.}
    \label{tour_de_delta}
\end{figure}

        \subsection{Inductive relation on the variance, first results}
Emme and Hubert \cite[Theorem 3.1]{JEPH1} give an explicit formula for the variance in base $2$. It is possible to  adapt their methods in order to find a similar expression in any base. However, here we just need some basic estimations about the variance. We first deduce from Proposition \ref{relat rec mu} an inductive relation on the variance.
\begin{lemme}\label{lemme_relat_rec_var}
For $\widetilde{r}\in\NN$ and $0\le r_0\le b-1$, we have the relation
        \begin{align*}
            \Var(\mu^{(b\widetilde{r}+r_0)})&=\frac{b-r_0}{b}\Var(\mu^{(\widetilde{r})})+\frac{r_0}{b}\Var(\mu^{(\widetilde{r}+1)})+r_0(b-r_0).
        \end{align*}
\end{lemme}
\begin{proof}
We compute using \eqref{esperance_nulle} and \eqref{relat rec mu}
\begin{align*}
    \Var(\mu^{(b\widetilde{r}+r_0)})  &=\sum_{d\in\ZZ} d^2 \mu^{(b\widetilde{r}+r_0)}(d)\\
                    &=\sum_{d\in\ZZ} d^2 \left(\frac{b-r_0}{b} \mu^{(\widetilde{r})} (d-r_0) + \frac{r_0}{b} \mu^{(\widetilde{r}+1)}(d+b-r_0)\right)\\
                    &=\frac{b-r_0}{b}\sum_{d'\in\ZZ}(d'+r_0)^2\mu^{(\widetilde{r})} (d') +\frac{r_0}{b} \sum_{d'\in\ZZ} (d'+b-r_0)^2 \mu^{(\widetilde{r}+1)}(d')
\end{align*}
Using again \eqref{esperance_nulle}, we get
\begin{align*}
    \Var(\mu^{(b\widetilde{r}+r_0)}) &= \frac{b-r_0}{b} \left(\Var(\mu^{(\widetilde{r})})+r_0^2 \right) + \frac{r_0}{b}\left(\Var(\mu^{(\widetilde{r}+1)})+(b-r_0)^2\right)\\
                    &= \frac{b-r_0}{b} \Var(\mu^{(\widetilde{r})}) + \frac{r_0}{b} \Var ( \mu^{(\widetilde{r}+1)}) + r_0(b-r_0).
\end{align*}
\end{proof}
Even if we do not need an explicit formula of $\Var(\mu^{(r)})$ for a general $r\in\NN$, we will need one in some specific cases. First, we are interested in the variance of $\mu^{(r)}$ when the expansion of $r$ is one digit long.

\begin{lemme}\label{formule_de_variance1} 
If $0\le r\le b-1$ then 
            \[\Var(\mu^{(r)})=r(1+b-r).\]
\end{lemme}
\begin{proof}
 We recall that $\mu^{(0)}=\delta_0$ so $\Var(\mu^{(0)})=0$. Then, if $r=1$, Lemma~\ref{lemme_relat_rec_var} gives
            \[ \Var(\mu^{(1)})=\frac{b-1}{b}\Var(\mu^{(0)})+\frac{1}{b}\Var(\mu^{(1)})+(b-1).\]
It follows
    \[\Var(\mu^{(1)})=b.\]
Finally, if $2\le r\le b-1$, again from Lemma \ref{lemme_relat_rec_var}, we have
            \[\Var(\mu^{(r)})=\frac{b-r}{b} \Var(\mu^{(0)})+\frac{r}{b}\Var(\mu^{(1)})+r(b-r)=r(1+b-r).\] 
\end{proof}

Now, we are interested in the variance of $\mu^{(r)}$ when the expansion of $r$ has a rightmost block of $(b-1)$'s of length $m\ge 1$ that is to say when there exists $\widehat{r}\in\NN$ such that $r=b^m\widehat{r}+b^m-1$.
\begin{lemme}\label{formule_de_variance4}
    For $\widehat{r}\in\NN$ and $m\ge 1$, the variance of $\mu^{(b^m\widehat{r}+b^m-1)}$ is 
    \[\Var(\mu^{(b^m\widehat{r}+b^m-1)})=\frac{1}{b^{m}}\Var(\mu^{(\widehat{r})})+\left(1-\frac{1}{b^{m}}\right)\Var(\mu^{(\widehat{r}+1)})+b-\frac{1}{b^{m-1}}.\]
\end{lemme}
\begin{remarque} 
We observe that, if we take $\widehat{r}=0$, then we find the variance of $\mu^{(r)}$ where the expansion of $r$ is composed of only one block of $(b-1)$'s.  Thus, with Lemma~\ref{formule_de_variance1}, we now have the exact value of the variance of $\mu^{(r)}$ when the expansion of $r$ is composed of exactly $1$ non-zero block: 
\begin{align}\label{variance pour un bloc}
   \Var(\mu^{(r)}) =&\left\{\begin{array}{ll}
         r(1+b-r) & \mbox{if } r=1,\cdots, b-2, \\
         2b-\frac{2}{b^{m-1}}& \mbox{if } r=b^m-1 \hspace{1 mm}(m\ge 1). \\
    \end{array}\right.
\end{align}
\end{remarque}
\begin{proof}
    We prove the lemma by induction on $m\ge 1$.
        \begin{enumerate}
        \item If $m=1$ then, from Lemma \ref{lemme_relat_rec_var}
        \begin{align*}
            \Var(\mu^{(b\widehat{r}+b-1)})&=\frac{1}{b} \Var(\mu^{(\widehat{r})})+\frac{b-1}{b}\Var(\mu^{(\widehat{r}+1)})+(b-1).
        \end{align*}
        That is what we want.
        \item If we assume that the lemma is true for $m-1$ then we consider the integer $b^m\widehat{r}+b^m-1$. We observe the trivial identity 
        \[b^m\widehat{r}+b^m-1=b\times\left(b^{m-1}\widehat{r}+\left(b^{m-1}-1\right)\right)+(b-1)\]
        It follows from Lemma~\ref{lemme_relat_rec_var}
        \begin{align*}
            \Var(\mu^{(b^m\widehat{r}+b^m-1)})&=\frac{1}{b} \Var\left(\mu^{\left(b^{m-1}\widehat{r}+\left(b^{m-1}-1\right)\right)}\right)+\frac{b-1}{b}\Var(\mu^{(b^{m-1}(\widehat{r}+1))})\\
            &\hspace{5mm}+(b-1).
        \end{align*}
        We use the induction hypothesis and the fact that $\mu^{(b^{m-1}(\widehat{r}+1))}=\mu^{(\widehat{r}+1)}$ (Lemma~\ref{relat rec mu})
        \begin{align*}
            \Var\left(\mu^{\left(b^{m-1}\widehat{r}+\left(b^{m-1}-1\right)\right)}\right)&=\frac{1}{b^{m-1}}\Var(\mu^{(\widehat{r})})+\left(1-\frac{1}{b^{m-1}}\right)\Var(\mu^{(\widehat{r}+1)})\\
            &\hspace{5mm} +b-\frac{1}{b^{m-2}}.
        \end{align*}
        Combining both gives the result for $m$.
    \end{enumerate}
\end{proof}

Finally, we consider the case where the expansion of $r$ has a units digit $1$, possibly with a block of $0$'s on its left. This corresponds to the existence of $\widehat{r}\in\NN$ and $m\ge 1$ such that $r=b^m\widehat{r}+1$.

\begin{lemme}\label{formule_de_variance3}
For $\widehat{r}\in\NN$ and $m\ge 1$, the variance of $\mu^{(b^m\widehat{r}+1)}$ is 
\[\Var(\mu^{(b^m\widehat{r}+1)})=\left(1-\frac{1}{b^m}\right)\Var(\mu^{(\widehat{r})})+\frac{1}{b^m}\Var(\mu^{\widehat{r}+1)})+b-\frac{1}{b^{m-1}}. \]
\end{lemme}
\begin{proof}
     We show by induction on $m\ge 1$.
    \begin{enumerate}
        \item If $m=1$ then, from Lemma \ref{lemme_relat_rec_var}
        \begin{align*}
            \Var(\mu^{(b\widehat{r}+1)})&=\frac{b-1}{b} \Var(\mu^{(\widehat{r})})+\frac{1}{b}\Var(\mu^{(\widehat{r}+1)})+(b-1).
        \end{align*}
        That is exactly what we want.
        \item If we assume that the formula is true for $m-1$, then we consider the integer $b^m\widehat{r}+1$. We have again by Lemma \ref{lemme_relat_rec_var}
        \begin{align*}
            \Var(\mu^{(b^m\widehat{r}+1)})&=\Var(\mu^{(b\times b^{m-1}\widehat{r}+1)})\\
            &=\frac{b-1}{b} \Var(\mu^{(b^{m-1}\widehat{r})})+\frac{1}{b}\Var(\mu^{(b^{m-1}\widehat{r}+1)})+(b-1).
        \end{align*}
        We use the induction hypothesis.
        \[\Var(\mu^{(b^{m-1}\widehat{r}+1)})=\left(1-\frac{1}{b^{m-1}}\right)\Var(\mu^{(\widehat{r})})+\frac{1}{b^{m-1}}\Var(\mu^{(\widehat{r}+1)})+b-\frac{1}{b^{m-2}}.\]
        Combining both gives the result for $m$.
    \end{enumerate}
\end{proof}

        \subsection{Upper and lower bound of the variance}\label{upper and lower bound of the variance}
Since $\Var(\mu^{(0)})=0$, the case $r=0$ is irrelevant and we suppose $r\ge 1$. We wish to find an upper and a lower bound of the variance of $\mu^{(r)}$ depending on $\rho(r)$, the number of blocks of $r$ defined in Definition \ref{defblock}. For convenient reasons, it is better to think in terms of non-zero blocks. We define, for an integer $r$, the quantity $\lambda(r)$ which corresponds to the number of non-zero blocks.
\begin{exemple} We use again the example of Figure \ref{blocktranslation}.
    \begin{figure}[H]
        \centering
        \includegraphics[scale=0.35]{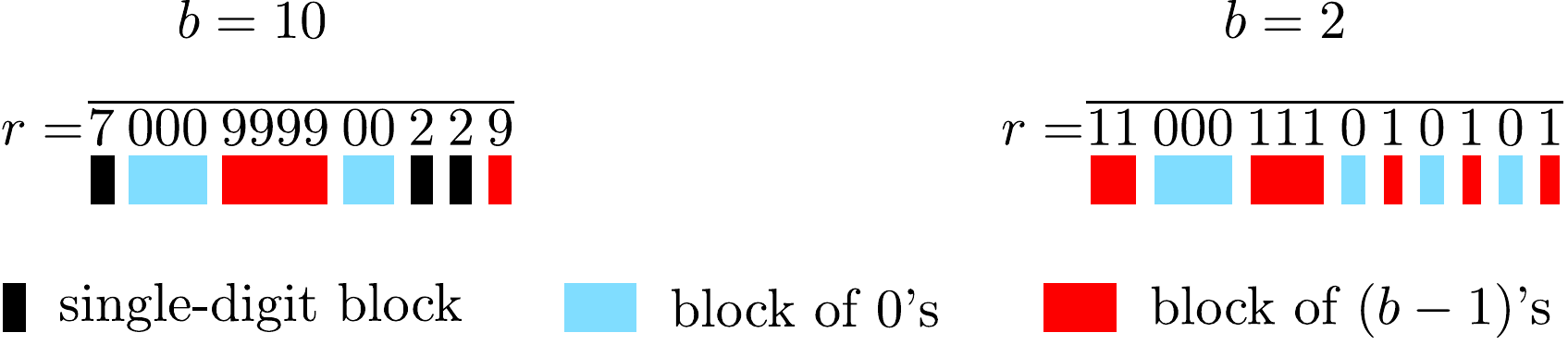}
        \caption{On the left-hand side, $\rho(r)=7$ and $\lambda(r)=5$. On the right-hand side, $\rho(r)=9$ and $\lambda(r)=5$.}
    \end{figure}
\end{exemple}
Of course, there is a relation between $\lambda(r)$ and $\rho(r)$ (we recall that $r\ge 1$):
\begin{align}\label{lien rho lambda}
    \lambda(r)\le \rho(r)\le 2\lambda(r).
\end{align}
We first give an upper bound of $\Var(\mu^{(r)})$ depending on $\lambda(r)$.
\begin{prop}\label{majoration}
For any $r\ge 1$ 
\begin{align}\label{majoration de la variance}
    \Var(\mu^{(r)})\le b^2  \lambda(r).
\end{align}
\end{prop}
We need the following lemma.
\begin{lemme}\label{comparaison}
For $r\ge 1$, we have the following inequality
    \[\lvert \Var(\mu^{(r+1)})-\Var(\mu^{(r)})\rvert\le b.\]
\end{lemme}

\begin{proof}
We recall that $r_0\in\{0, \cdots,b-1\}$ is the units digit of $r$ and the relation $r=b\widetilde{r}+r_0$. Let us prove
\[\lvert\Var(\mu^{(r+1)})-\Var(\mu^{(r)})\rvert\le \frac{\lvert\Var(\mu^{(\widetilde{r}+1)})-\Var(\mu^{(\widetilde{r})})\rvert}{b}+ b-1.\]
Indeed, there are two cases.
\begin{enumerate}
    \item If $r_0=b-1$, then $r=b\widetilde{r}+(b-1)$ and using twice Lemma~\ref{lemme_relat_rec_var}, we get
    \begin{align*}
        \Var(\mu^{(r+1)})-\Var(\mu^{(r)})&=\Var(\mu^{(b(\widetilde{r}+1))})-\Var(\mu^{(b\widetilde{r}+b-1)})\\&=\frac{\Var(\mu^{(\widetilde{r}+1)})-\Var(\mu^{(\widetilde{r})})}{b}-(b-1).
    \end{align*}
    \item If $r_0\in\{0, \cdots, b-2\}$ then, using the same tools, we compute 
    \begin{align*}
        \Var(\mu^{(r+1)})-\Var(\mu^{(r)})&=\Var(\mu^{(b\widetilde{r}+r_0+1)})-\Var(\mu^{(b\widetilde{r}+r_0)})\\&=\frac{\Var(\mu^{(\widetilde{r}+1)})-\Var(\mu^{(\widetilde{r})})}{b}+b-2r_0-1.
    \end{align*}
    We observe that $\lvert b-2r_0-1\rvert\le b-1$.
\end{enumerate}
Then, we conclude by an easy induction on the number of digits of $r$ and by checking that
\begin{align*}
    \lvert \Var(\mu^{(r+1)})-\Var(\mu^{(r)})\rvert &\le \frac{\lvert\Var(\mu^{(\widetilde{r}+1)})-\Var(\mu^{(\widetilde{r})})\rvert}{b}+ b-1 \\
        &\le \frac{b}{b}+b-1=b.
\end{align*}
\end{proof}
\begin{proof}[Proof of Proposition~\ref{majoration}]
Observe that we have
\begin{subequations}
    \begin{align*}
        b^2 &\ge j(1+b-j) \quad \mbox{ for all }  j=1, \cdots b-1, \tag{$C_0$} \label{C_7}\\
        b^2 &\ge  2b, \tag{$C_1$}\label{C_8}\\
        b^2 &\ge jb \quad \mbox{ for all } j=0, \cdots, b-1. \tag{$C_2$}\label{C_9}\\
    \end{align*}
\end{subequations}

We proceed by induction on $\lambda(r)\ge 1$.

\begin{enumerate}
    \item Initialisation: if $\lambda(r)=1$ then we have two cases depending on the type of blocks we are considering. But the variance is given at \eqref{variance pour un bloc}. Using Conditions~\eqref{C_7} and \eqref{C_8}, we can deduce that in both cases we have $\Var(\mu^{(r)})\le b^2 $.
    \item Inductive step: we let $n>1$ and we assume that if $\lambda(r)\le n$ then $\Var(\mu^{(r)})\le b^2 \lambda(r)$. We now assume that our $r\in\NN$ satisfies $\lambda(r)=n+1$. Let $\ell\ge 0$, $\widehat{r}\in\NN$ and $B_1$ a non-zero block such that we can write the expansion of $r$ as follow 
    \[\overline{\widehat{r} \hspace{1mm} \hspace{1mm} B_1 \hspace{1mm} 0^{\ell}}.\]
     We can assume $\ell=0$ but $\widehat{r}$ may have a rightmost block composed of $0$'s. We discuss on the type of $B_1$.
     \begin{enumerate}
        \item If $B_1$ is a block of ($b-1$)'s of length $m$ then $r=b^{m}\widehat{r}+b^{m}-1$ and we have the trivial equality
        \begin{align*}
            \Var(\mu^{(r)}) &= \Var(\mu^{(\widehat{r})})+\left(\Var(\mu^{(r)})-\Var(\mu^{(r+1)})\right)\\
                            & \quad -\left(\Var(\mu^{(\widehat{r})})-\Var(\mu^{(\widehat{r}+1)})\right).
        \end{align*}
        Using Lemma \ref{comparaison}, the inductive hypothesis and Condition \eqref{C_8}
         \[\Var(\mu^{(r)})\le b^2 n + b +b \le b^2 (n+1).\]
         \item If $B_1$ is a single-digit block $r_0$ then $r=b\widehat{r}+r_0$ and with we have another trivial equality 
         \begin{align*}
             \Var(\mu^{(r)})&=\Var(\mu^{(\widehat{r})}) + \sum_{k=1}^{r_0} \Var(\mu^{(b\widehat{r}+k)})-\Var(\mu^{(b\widehat{r}+k-1)}).
         \end{align*}
         Then, using Lemma \ref{comparaison}, the inductive hypothesis and Condition \eqref{C_9}, we find
         \[\Var(\mu^{(r)})\le b^2 n + r_0b \le b^2 (n+1).\]
     \end{enumerate}
     This concludes the inductive step and the proof.
\end{enumerate}
\end{proof}

We also have a lower bound depending on $\lambda(r)$.

\begin{prop}\label{minoration}
For any $r\ge 1$
\begin{align}
    \Var(\mu^{(r)})&\ge \frac{b}{4}  \lambda(r).\label{formule_minoration}
\end{align}
\end{prop}

\begin{proof}
Observe that 
\begin{subequations}
    \begin{align}
        \frac{b}{4} &\le \min\{j(b-j): j=1, \cdots b-1\}, \tag{$C_3$} \label{C_0}\\
        \frac{b}{4} &\le  (b-1), \tag{$C_4$}\label{C_1}\\
        \frac{b}{4} &\le \min \{j(b-j-\frac{1}{2}): j=1, \cdots b-1\}, \tag{$C_5$}\label{C_2}
    \end{align}
\end{subequations}
Of course, some of these conditions are redundant but we keep them all for simplicity because each one will be used in the proof.
We prove the result using induction on $\lambda(r)\ge 1$.
\begin{enumerate}
    \item Initialisation: if $\lambda(r)=1$ then we have two cases depending on the type of blocks we are considering. However, using \eqref{variance pour un bloc}, Conditions~\eqref{C_0} and \eqref{C_1}, we can deduce $\Var(\mu^{(r)})\ge \frac{b}{4} $.
     \item  Inductive step: we let $n>1$ and we assume that if $\lambda(r)\le n$ then $\Var(\mu^{(r)})\ge \frac{b}{4}\lambda(r)$. We now take $r\in\NN$ such that $\lambda(r)=n+1$, and we write
     \[r=b\widetilde{r}+r_0.\]
     Without loss of generality, we can assume that $r_0\neq 0$. Indeed, if $r_0=0$ then $\Var(\mu^{(r)})=\Var(\mu^{(\widetilde{r})})$ and $\lambda(r)=\lambda(\widetilde{r})$.
     
     From  Lemma \ref{lemme_relat_rec_var}, we get
     \begin{align}\label{relat rec  variance}
     \Var(\mu^{(r)})=\frac{b-r_0}{b}\Var(\mu^{(\widetilde{r})})+\frac{r_0}{b}\Var(\mu^{(\widetilde{r}+1)})+r_0(b-r_0).
     \end{align}
     We discuss about the value of $\lambda(\widetilde{r})$ which can be either $n$ or $n+1$.
     \begin{enumerate}
         \item If $\lambda(\widetilde{r})=n$, which means that $r_0\neq b-1$ or $r_1\neq b-1$: we use the induction hypothesis to get $\Var(\mu^{(\widetilde{r})})\ge \frac{b}{4}\lambda(\widetilde{r})=\frac{bn}{4}$, and it follows that 
         \[\Var(\mu^{(r)})\ge\frac{b-r_0}{4}n+\frac{r_0}{b}\Var(\mu^{(\widetilde{r}+1)})+r_0(b-r_0).\]
         We now observe that $n-2\le\lambda(\widetilde{r}+1)\le n+1$. The reader is referred to Figure \ref{classification} (with $\widetilde{r}$ instead of $r$) for more details. We consider two cases.
         \begin{enumerate}
             \item If $\lambda(\widetilde{r}+1)=n+1$ then we cannot apply the induction hypothesis. In this case $\widetilde{r}$ is a multiple of $b$ if $b\ge 3$ and even $b^2$ when $b=2$ (see Figure \ref{classification}). So, there exists $\widehat{r}\in\NN$ and $m\ge 1$ (or $2$ if $b=2$) such that  
             \[\widetilde{r}=b^m\widehat{r} \hspace{5mm} \mbox{ and } \hspace{5mm} b{\not|}\,\widehat{r}.\]
             
             which means $\lambda(\widehat{r})=\lambda(\widetilde{r})=n$ and $\lambda(\widehat{r}+1)< n+1$ since the rightmost digit of $\widehat{r}$ is not $0$ (see Figure~\ref{classification}). We can apply Lemma \ref{formule_de_variance3}
             \[\Var(\mu^{(\widetilde{r}+1)})=\left(1-\frac{1}{b^m}\right)\Var(\mu^{(\widehat{r})})+\frac{1}{b^m}\Var(\mu^{(\widehat{r}+1)})+b-\frac{1}{b^{m-1}}. \]
             
             Since $\mu^{(\widetilde{r})}=\mu^{(\widehat{r})}$ and $\lambda(\widetilde{r})=n$, we deduce from the induction hypothesis and \eqref{relat rec  variance} that
             \begin{align*}
                 \Var(\mu^{(r)})
                    &\ge\frac{b-r_0}{4}n+\frac{r_0}{b}\left[\frac{b}{4}n+b-\frac{3}{2b^{m-1}}\right]+\frac{b}{4}\\
                    &\ge\frac{b}{4}(n+1)+r_0\left(1-\frac{3}{2b^{m-1}}\right)\\
                    &\ge \frac{b}{4}(n+1).
             \end{align*}
             \item If $\lambda(\widetilde{r}+1)\neq n+1$ 
              then we can apply the induction hypothesis.
            \begin{align*}
                 \Var(\mu^{(r)})&\ge\frac{b-r_0}{4}n+\frac{r_0}{2}(n-2)+r_0(b-r_0) \\
                    &\ge \frac{b}{4} n -\frac{r_0}{2}  +r_0(b-r_0).
             \end{align*}
             Thanks to \eqref{C_2}
             \[r_0(b-r_0)-\frac{r_0}{2} \ge \frac{b}{4}.\]
             And so
             \[\Var(\mu^{(r)})\ge \frac{b}{4}n+\frac{b}{4} =\frac{b}{4} (n+1).\]

         \end{enumerate}
         We conclude the case $\lambda(\widetilde{r})=n$.
         \item If $\lambda({\widetilde{r}})=n+1$: it means that the rightmost block in the expansion of $r$ is a block of $(b-1)$'s of length $m\ge 2$. So, there exists $\widehat{r}\in\NN$ such that $r=b^m\widehat{r}+b^m-1$ and the rightmost digit in the expansion of $\widehat{r}$ is not $(b-1)$, that is to say $b{\not|}\, \widehat{r}+1$. We are in the context of Lemma \ref{formule_de_variance4}, we have 
         \[\Var(\mu^{(r)})=\frac{1}{b^{m}}\Var(\mu^{(\widehat{r})})+\left(1-\frac{1}{b^{m}}\right)\Var(\mu^{(\widehat{r}+1)})+b-\frac{1}{b^{m-1}}.\]
         Since $\lambda(\widehat{r})=n$ and $n-1\le \lambda(\widehat{r}+1) \le n+1$ ($\widehat{r}$ does not start with a block of $(b-1)$'s so $\lambda(\widehat{r}+1)\neq n-2$, see Figure \ref{classification}), we have now
         \begin{align}\label{formule_dans_preuve_minoration}
             \Var(\mu^{(r)})&\ge\frac{n}{4b^{m-1}}+\left(1-\frac{1}{b^{m}}\right)\Var(\mu^{(\widehat{r}+1)})+b-\frac{1}{b^{m-1}}.
         \end{align}
         We discuss about the possible values of $\lambda(\widehat{r}+1)$.
         \begin{enumerate}
            \item If $\lambda(\widehat{r}+1)=n+1$ then it is just as in the point (a)i. of this proof, it means that $\widehat{r}$ starts with a block of $0$'s. So we let $m'\ge 1$ ($2$ if $b=2$) and $\widehat{\widehat{r}}\in\NN$ such that $\widehat{r}=b^{m'}\widehat{\widehat{r}}=$ and $b{\not|}\,\widehat{\widehat{r}}$. We are again in the context of Lemma \ref{formule_de_variance3} 
            \[\Var(\mu^{(\widehat{r}+1)})=\left(1-\frac{1}{b^{m'}}\right)\Var(\mu^{(\widehat{\widehat{r}})})+\frac{1}{b^{m'}}\Var(\mu^{(\widehat{\widehat{r}}+1)})+b-\frac{1}{b^{m'-1}}.\]
            We observe $\lambda(\widehat{\widehat{r}})=\lambda(\widehat{r})=n$ and $n-2\le\lambda(\widehat{\widehat{r}}+1)\le n$ (again, it cannot be $n+1$ because $\widehat{\widehat{r}}$ does not start with a block of $0$'s).
            
            So we have
            \begin{align*}
                \Var(\mu^{(\widehat{r}+1)}) &\ge \left(1-\frac{1}{b^{m'}}\right)\frac{b}{4} n+\frac{1}{4b^{m'-1}} (n-2)+b-\frac{1}{b^{m'-1}}\\
                    &\hspace{5mm}=\frac{b}{4} n +b-\frac{3}{2b^{m'-1}}.
            \end{align*}
            From \eqref{formule_dans_preuve_minoration} we deduce
            \begin{align*}
                \Var(\mu^{(r)}) &\ge \frac{b}{4} n+\left(1-\frac{1}{b^m}\right)\left(b-\frac{3}{2b^{m'-1}}\right)+b-\frac{1}{b^{m-1}}. 
            \end{align*}
            We observe that
            \[b-\frac{3}{2b^{m'-1}}\ge 0\]
            as well as
            \[b-\frac{1}{b^{m-1}}\ge \frac{b}{4}.\]
            So we can write 
            \[\Var(\mu^{(r)})\ge \frac{b}{4} (n+1).\]
            \item If $\lambda(\widehat{r}+1)\neq n+1$ then 
            we can apply the induction hypothesis.
                \begin{align*}
                    \Var(\mu^{(r)}) &\ge\frac{n}{4b^{m-1}}+ \left(1-\frac{1}{b^{m}}\right)\frac{b}{4} (n-1)+b-\frac{1}{b^{m-1}}\\
                    &\ge \frac{b}{4} n-\left(1-\frac{1}{b^{m}}\right)\frac{b}{4} +b-\frac{1}{b^{m-1}}\\
                \end{align*}
                We observe that 
                \[b-\frac{1}{b^{m-1}}-\left(1-\frac{1}{b^{m}}\right)\frac{b}{4} \ge \frac{b}{4},\]
                so 
                \[\Var(\mu^{(r)})\ge \frac{b}{4} n +\frac{b}{4} =\frac{b}{4} (n+1).\]
         \end{enumerate}
         It concludes the case $\lambda(\widetilde{r})=n+1$. The statement is thus true when $\lambda(r)=n+1$.
     \end{enumerate}
\end{enumerate}
\end{proof}
The following figure shows how the number of blocks behaves of an integer when we add $1$ to it.
\begin{figure}[H]
    \centering
    \includegraphics[scale=0.32]{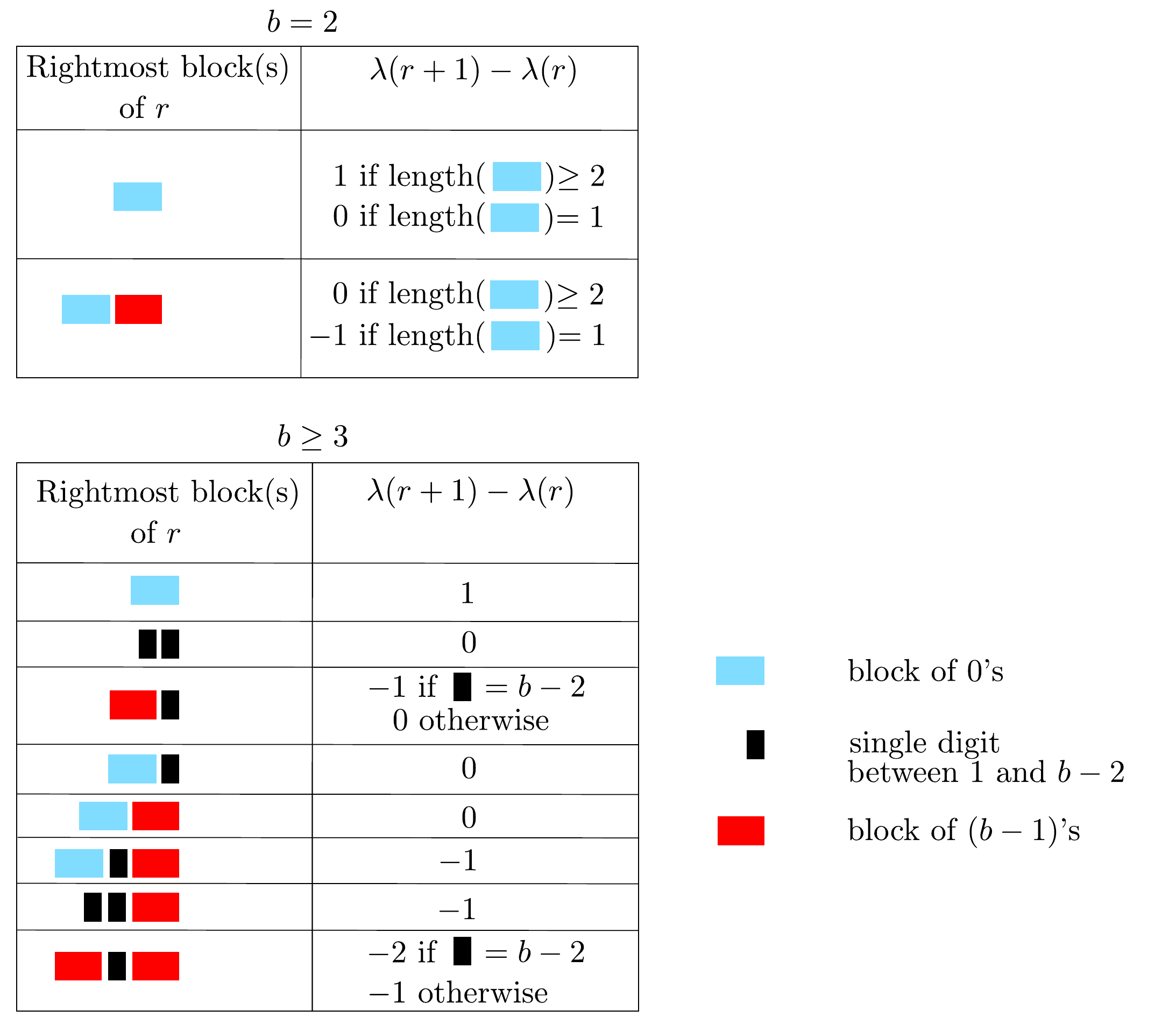}
    \caption{Variations of the number of non-zero blocks when we add $1$.}
    \label{classification}
\end{figure}
\begin{proof}[Proof of Theorem~\ref{encadrement_de_la_variance}]
We use Proposition~\ref{majoration}, Proposition~\ref{minoration} and \eqref{lien rho lambda} to get the result.
\end{proof}

\section{A \texorpdfstring{$\phi$}{P}-mixing process}\label{A phi-mixing process}
We work on the probability space $(\XX,\mathscr{B}(\XX),\PP)$. For a given integer $r$, $\Delta^{(r)}$ is viewed as a random variable with law $\mu^{(r)}$ by Corollary~\ref{corollaire_prop_moment} (the randomness comes from the argument $x$ of $\Delta^{(r)}$, considered as a random outcome in $\XX$ with law $\PP$).
Our purpose in this section is to study the asymptotic behaviour of $\mu^{(r)}$ as the number of blocks $\rho(r)$ goes to infinity. For this, we will decompose $\Delta^{(r)}$ as a sum
\[\Delta^{(r)}=\sum_{i=1}^{\lambda(r)}X_i^{(r)}\]
where $\lambda(r)$ is the number of non-zero blocks in the base-$b$ expansion of  $r$ (see Section~\ref{upper and lower bound of the variance}), and $(X_1^{(r)},\cdots,X_{\lambda(r)}^{(r)})$ is a finite process defined in the next section. We will prove and use some mixing properties of this process to get our result.
    \subsection{The process}\label{The process}
Again, the case $r=0$ is irrelevant so we assume $r\ge 1$. For $1\le i\le \lambda(r)$, we will write $B_i$ as the $i^{th}$ non-zero block present in the expansion of $r$, starting from the left-hand side of the expansion and ending at the units digit. We now define $r[i]$ as the integer whose base-$b$ expansion is obtained as follows:  for $k=i+1, \cdots, \lambda(r)$, we replace the block $B_k$ by a block of $0$'s of the same length (see Figure \ref{construction r[i]} below). We observe that $r[\lambda(r)]=r$.  
\begin{figure}[H]
    \centering
    \includegraphics[scale=0.34]{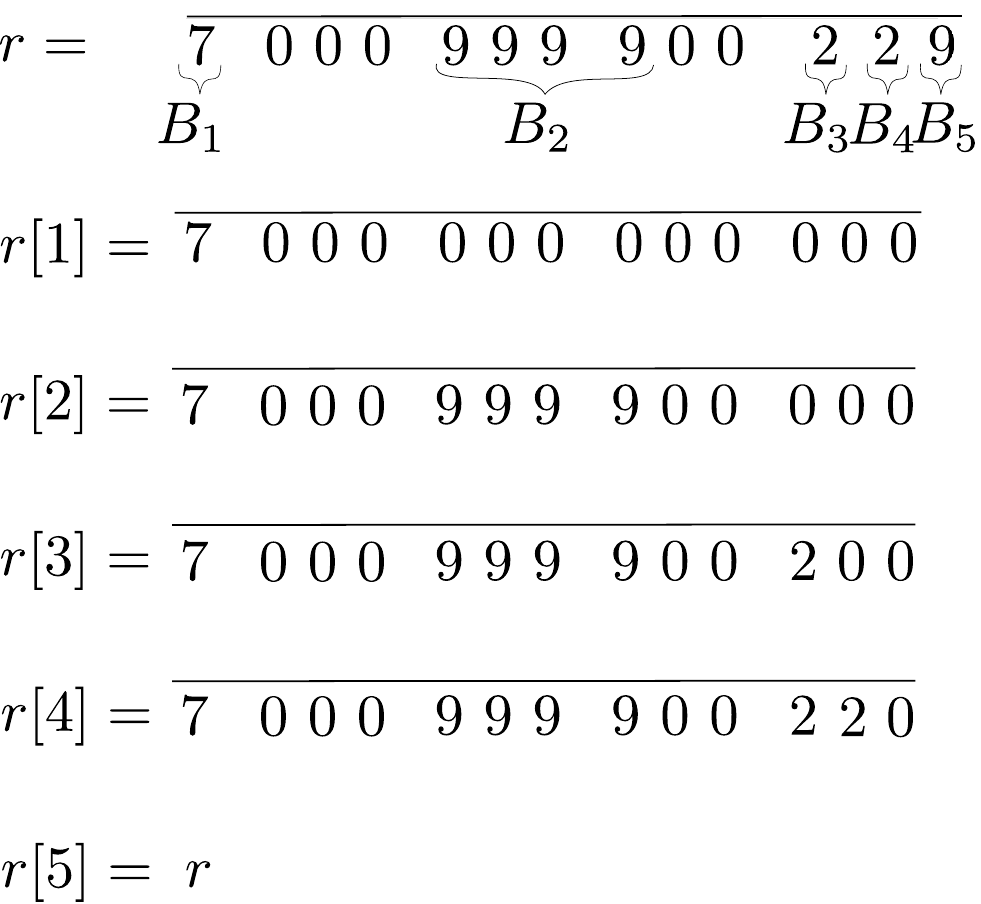}
    \caption{Example in base $b=10$.}
    \label{construction r[i]}
\end{figure} 
With the convention $r[0]:=0$, we observe the trivial equality
\[r=\sum_{i=1}^{\lambda(r)}r[i]-r[i-1].\]
For $1\le i\le \lambda(r)$, we define almost everywhere on $\XX$ (see Subsection~\ref{Sum of digits on the odometer})
\[X_i^{(r)}:=\Delta^{(r[i]-r[i-1])}\circ T^{r[i-1]}.\] 
Since $r[i]-r[i-1]=\overline{B_i 0\cdots 0}$, the function  $X_i^{(r)}$ is a random variable corresponding to the action of the $i^{th}$ block $B_i$ once the previous blocks have already been taken into consideration.
From \eqref{decomp_sum}, we deduce
\[\Delta^{(r)}=\sum_{i=1}^{\lambda(r)}X_i^{(r)}.\]
In particular, if $x\in\XX$ is randomly chosen with law $\PP$, then $\displaystyle\sum_{i=1}^{\lambda(r)}X_i^{(r)}(x)$ follows the law $\mu^{(r)}$.
Hence, the standard deviation $\sigma_r$ of $\mu^{(r)}$ defined in Theorem~\ref{theoprincipal} satisfies
\[\sigma_r^2=\Var\left(\sum_{i=1}^{\lambda(r)}X_i^{(r)}\right).\]
We first show that every moment of $X_i^{(r)}$ is bounded from above by a constant independent of $r$ and $i$.
\begin{lemme}\label{moments bornées de X_i}
For every $k\in\NN$, there exists a constant $C_k>0$ such that
\[\forall r\in\NN,\, \forall 1\le i \le \lambda(r), \quad \quad \EE\left(\left|X_i^{(r)}\right|^k\right)\le C_k.\]
\end{lemme}
\begin{proof}
Let $k\in\NN$. If $X_i^{(r)}$ corresponds to the action of a single-digit block, that is the action of one digit $\alpha$ between $1$ and $b-2$, then its law is given by $\mu^{(\alpha)}$ whose moments are all finite (see Corollary~\ref{corollaire_prop_moment}). So, we have in this case
\[\EE\left(\left|X_i^{(r)}\right|^k\right)\le\max\left\{\EE\left(\left|\Delta^{(\alpha)}\right|^k\right): 1\le\alpha\le b-2\right\}.\]
Now, if $X_i^{(r)}$ corresponds to the action of a block of $(b-1)$'s, there exist two integers $n\le m$ such that $r[i]-r[i-1]=b^m-b^n$.
It follows from the definition of $X_i^{(r)}$ and \eqref{decomp_sum} that
\[X_i^{(r)}=\Delta^{(b^m-b^n)}\circ T^{b^n}\overset{d}{=}\Delta^{(b^m-b^n)}=\Delta^{(b^m)}-\Delta^{(b^n)}\circ T^{b^m-b^n}\]
where $\overset{d}{=}$ means the equality in distribution. Since $\mu^{(b^n)}=\mu^{(b^m)}=\mu^{(1)}$, we can write $X_i^{(r)}$ as the difference of two dependent random variables following the law $\mu^{(1)}$ which has finite moments. So there exists a constant that depends only on $k$ such that $\EE\left(|X_i^{(r)}|^k\right)$ is bounded by this constant.
\end{proof}
The next part is devoted to the estimation of the so-called $\phi$-mixing coefficients for the finite sequence $(X_i^{(r)})_{1\le i\le \lambda(r)}$.

    \subsection{The \texorpdfstring{$\phi$}{P}-mixing coefficients}
There exist many types of mixing coefficients (see e.g. the survey \cite{RCB1} by Bradley). Those we are working with are commonly called ``$\phi$-mixing coefficients''.
\begin{definition}
   Let $(X_i)_{i\ge 1}$ be a (finite or infinite) sequence of random variables. The associated \emph{$\phi$-mixing coefficients} $\phi(k)$, $k\ge 1$, are defined by
   \[\phi(k):=\sup_{p\ge 1} \hspace{1mm} \sup_{A,B} \hspace{2mm} \lvert \PP_A(B)-\PP(B) \rvert\]
   where the second supremum is taken over all events $A$ and $B$ such that 
\begin{itemize}
    \item $A\in\sigma(X_i:1\le  i\le p)$,
    \item $\PP(A)>0$ and 
    \item $B\in\sigma(X_i:  i\ge k+p)$.
\end{itemize}
By convention, if $X_i$ is not defined when $i\ge k+p$ then the $\sigma$-algebra is trivial.
\end{definition}
In the case of a finite sequence $(X_1,\cdots,X_n)$, the convention implies that $\phi(k)=0$ for $k\ge n$.
We now give an upper bound on the $\phi$-mixing coefficients for the process $(X_i^{(r)})$ defined in Section \ref{The process}.
\begin{lemme}\label{lemme_de_melange}
  For $r\ge 1$, the mixing coefficients of $(X_i^{(r)})_{1\le i\le \lambda(r)}$ satisfy 
  \[\forall k\ge 1, \quad \phi(k)\le 2\left(\frac{b-1}{b}\right)^{\frac{k}{2}-1}.\]
\end{lemme}
  
\begin{proof}
Let $k$ and $p$ be two integers. We observe that if $k=1,2$, the inequality is trivial so we assume that $k\ge 3$. We call \emph{buffer strip} the set of indices corresponding to the positions of the digits between $B_p$ and $B_{k+p}$ (both excluded). It depends on $r$, $p$ and $k$ so we denote it by $\mathcal{I}_{r,p,k}$. We consider the event
\[C:=\{x\in\XX : \exists j\in\mathcal{I}_{r,p,k} \mbox{ with }r_j<b-1 \mbox{ such that } x_j=0\}.\]
We are going to show that, for $k$ large enough, $C$ is a high-probability event and that, conditionned to $C$, two events  $A\in\sigma\left(X_i^{(r)}: 1\le i\le p\right)$ and $B\in\sigma\left(X_i^{(r)}: i\ge k+p\right)$ are always independent. \begin{figure}[H]
    \centering
    \includegraphics[scale=0.35]{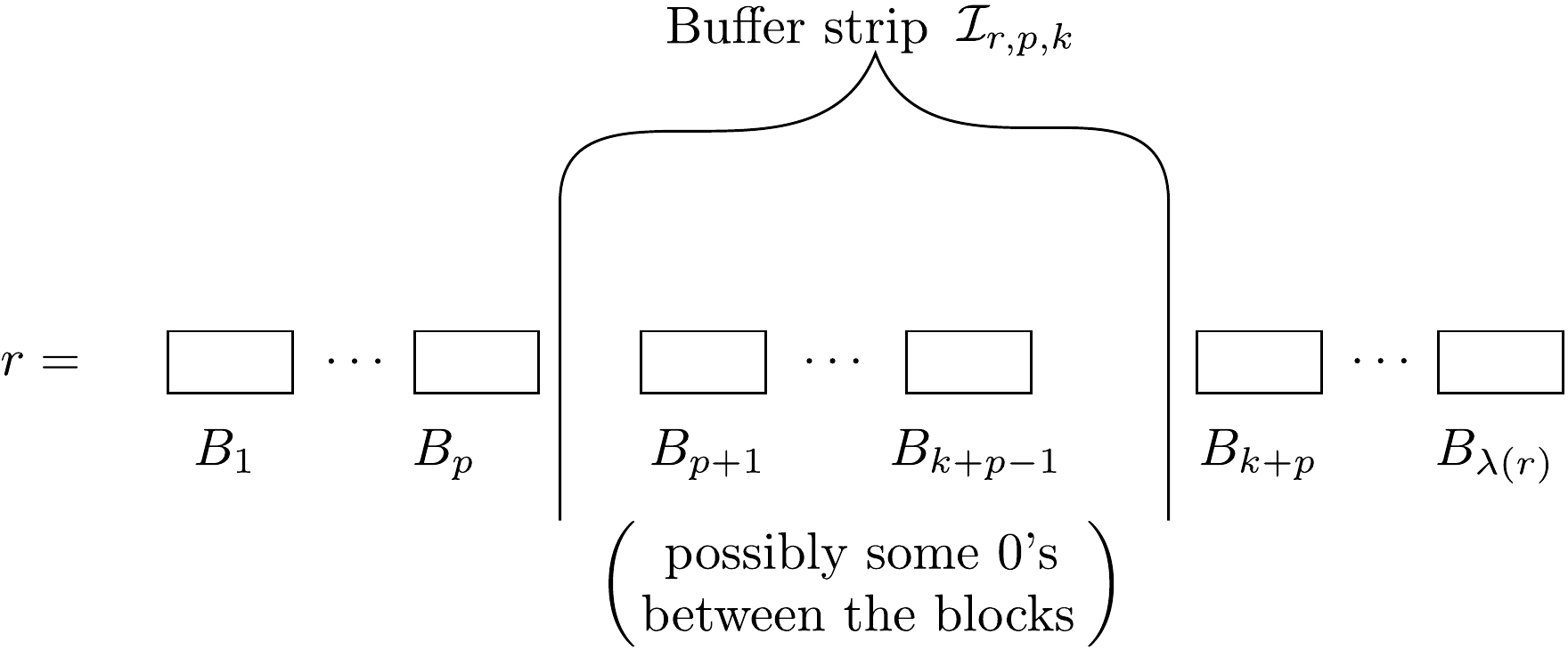}
    \caption{Visualization of the buffer strip.}
\end{figure}
Denoting by $\overline{C}$ the complement of $C$ in $\XX$, we have
\begin{align}\label{proba_de_c_compl}
    \PP(\overline{C})&=\left(\frac{b-1}{b}\right)^t
\end{align}
where $t:=\Big|\{j\in\mathcal{I}_{r,p,k}\mbox{ : } r_j\neq b-1\}\Big|$.
We are going to show that 
\begin{align}\label{inegalite_sur_t}
    t&\ge \frac{k}{2}-1.
\end{align}
Indeed, there are $k-1$ non-zero blocks in the buffer strip. Let $\ell\in\NN$ be the number of blocks of ($b-1$)'s. Then there are $k-1-\ell$ single-digit blocks. There are two cases.
\begin{enumerate}
    \item If $\ell\le \frac{k}{2}$, then $k-1-\ell\ge \frac{k}{2}-1$ and we get \eqref{inegalite_sur_t}.
    \item Otherwise, $\ell> \frac{k}{2}$. Since the blocks of ($b-1$)'s are separated using blocks of zeros or single-digit blocks, there are at least $\frac{k}{2}-1$ blocks of zeros of single-digit blocks, which also yields \eqref{inegalite_sur_t}.
\end{enumerate}
So, we get from \eqref{proba_de_c_compl} and \eqref{inegalite_sur_t} that
\begin{align*}
    \PP(C)\ge 1-\left(\frac{b-1}{b} \right)^{\frac{k}{2}-1}>0.
\end{align*}
Now, let $A\in\sigma\left(X_i^{(r)}: 1\le i\le p\right)$ with $\PP(A)>0$ and $B\in\sigma\left(X_i^{(r)}: i\ge k+p\right)$. 

Observe that $A$ and $C$ are independent. Indeed, $C$ only depends on indices in $\mathcal{I}_{r,p,k}$ while $A$, by construction of the random variables $(X_i^{(r)})_{1\le i\le \lambda(r)}$, only depends on the subset of indices on the left-hand side of the buffer strip. We deduce 
\begin{align}\label{indep_A&C}
    \PP(A\cap C)=\PP(A)\PP(C)>0.
\end{align}

Observe also that, conditionned to $C$, $A$ and $B$ are independent. Indeed, when $C$ is realized, there exists an index $j$ in $\mathcal{I}_{r,p,k}$ such that $r_j\neq b-1$ and $x_j=0$. At this position, a carry cannot be created and, furthermore, a carry propagation coming from the right-hand side will be stopped at this index. In other words, when $C$ is realized, the carries created by the blocks $B_i$, $i\ge k+p$, never spread on the left-hand side of the buffer strip. Moreover, we deduce that $A$ and $B\cap C$ are independent. Indeed, we compute
\begin{align}\label{indep_A&(B&C)}
    \PP(A\cap B\cap C)&=\PP_C(A\cap B)\PP(C) \notag\\
                      &=\PP_C(A)\PP_C(B)\PP(C) \notag \\
                      &=\PP(A)\PP(B\cap C).
\end{align}
Now, we have
\begin{align}\label{inegalite_triang}
    |\PP_A(B)-\PP(B)| &\le |\PP_A(B)-\PP_{A\cap C}( B)| +|\PP_{A\cap C}(B)-\PP(B)|. 
\end{align}
But, using \eqref{indep_A&C} and \eqref{indep_A&(B&C)}, we obtain
\[\PP_{A\cap C}( B)=\frac{\PP(A\cap B\cap C)}{\PP(A\cap C)}=\frac{\PP(B\cap C)
}{\PP(C)}=\PP_C(B).\]
So, \eqref{inegalite_triang} becomes
\begin{align}\label{inegalite_triang2}
    |\PP_A(B)-\PP(B)| &\le |\PP_A(B)-\PP_{A\cap C}( B)| +|\PP_{C}(B)-\PP(B)|. 
\end{align}
Now, one can show that for any event $D$, 
\[\big|\PP(D)-\PP_C(D)\big| \le \PP\left(\overline{C}\right).\]
This general inequality is also true replacing $\PP$ by $\PP_A$. We observe that the measure $\PP_A$ conditionned to $C$ is the measure $\PP_{A\cap C}$. So, coming back to \eqref{inegalite_triang2}, we get
\[|\PP_A(B)-\PP(B)|\le \PP_A\left(\overline{C}\right)+\PP\left(\overline{C}\right)= 2\PP\left(\overline{C}\right) \le 2\left(\frac{b-1}{b}\right)^{\frac{k}{2}-1}.\]
\end{proof} 


\section{Proof of Theorem~\ref{theoprincipal}}

    \subsection{A result from Sunklodas and first step of the proof}

In this section, we state a result by Sunklodas \cite{JKS} about the speed of convergence in the Central Limit Theorem for $\phi$-mixing sequences. Actually, our formulation is new but the proof is an immediate consequence of \cite[Theorem 1]{JKS}.

We first need to introduce some notations. Let $Y$ be a standard normal random variable. Consider $\xi_1, \cdots, \xi_n$ a finite sequence of $n$ random variables with $\phi$-mixing coefficients $\phi(k)$ for $k=1,2, \cdots$. Write 
\[V:=\sqrt{\Var\left(\sum_{i=1}^n\xi_i\right)} \hspace{5mm} \mbox{ and } \hspace{5mm} Z:=\sum_{i=1}^n \frac{\xi_i}{V}.\]
We need to add, for technical reason, some other notations.
\[\Phi_{1/2}:=\sum_{k\ge 1} k
\sqrt{\phi(k)} \hspace{5mm} \mbox{and} \hspace{5mm} \overline{\Phi_{1/2}}:=\max\left\{\sqrt{\Phi_{1/2}},\Phi_{1/2}^2\right\}.\]

We observe that $\Phi_{1/2}$ is defined by a sum on, actually, a finite number of non-zero terms because we are in the case of a finite sequence of random variables. Our formulation of \cite[Theorem 1]{JKS} is the following.
\begin{theo}\label{sunklodas}
Assume for $i=1, \cdots, n$ that $\EE(\xi_i)=0$ and $\EE(\xi_i^4)<\infty$. Let $h:\RR\rightarrow\RR$ be a thrice differentiable function such that $||h^{'''}||_\infty<\infty$. Then
\begin{align*}
    \bigg|\EE\left(h(Z)-h(Y)\right)\bigg| & \le ||h^{'''}||_{\infty}\left(\frac{5}{2}+28\overline{\Phi_{1/2}}\right) \sum_{i=1}^n \EE\left(\left|\frac{\xi_i}{V}\right|^3\right)\\
    &\hspace{2mm} +120||h^{'''}||_{\infty}\overline{\Phi_{1/2}}\sqrt{\sum_{i=1}^n\EE\left(\left|\frac{\xi_i}{V}\right|^2\right)}\sqrt{\sum_{i=1}^n\EE\left(\left|\frac{\xi_i}{V}\right|^4\right)}.
\end{align*}
\end{theo}
We are going to show that we can apply this result to prove Theorem~\ref{theoprincipal}. We start with \eqref{vitesse_fonction_reguliere}.
\begin{proof}[Proof of \eqref{vitesse_fonction_reguliere}]
Let $r\in\NN^*$. The variables $(X_i^{(r)})_{1\le i\le \lambda(r)}$ are of zero-mean and have a finite moment of order $4$. Lemma \ref{lemme_de_melange} gives a universal upper bound for $\Phi_{1/2}$. So we can apply Theorem~\ref{sunklodas} to the sequence $(X_i^{(r)})_{1\le i\le \lambda(r)}$. We observe that, from \eqref{lien rho lambda}, Theorem~\ref{encadrement_de_la_variance} and Lemma~\ref{moments bornées de X_i}, for $j=2,3,4$
\[\sum_{i=1}^{\lambda(r)} \EE\left(\left|\frac{X_i^{(r)}}{\sigma_r}\right|^j\right)=\frac{C_j\lambda(r)}{\sigma_r^j}\le \frac{2^jC_j}{b^{\frac{j}{2}}} \rho(r)^{1-\frac{j}{2}}.\]
So, by Theorem~\ref{sunklodas}, there exists $K>0$ such that
\begin{align}\label{majo_fonction_reg}
    \left|\EE\left(h\left(\sum_{i=1}^{\lambda(r)} \frac{X_i^{(r)}}{\sigma_r}\right)-h(Y)\right)\right|
    & \le \frac{K||h'''||_{\infty}}{\sqrt{\rho(r)}} 
\end{align}
\end{proof}
It remains to prove \eqref{vitesse indicatrice}.

    \subsection{Speed of convergence of the cumulative distribution functions}
    
Observe that for all $t\in\RR$
\[F_r(t)-F(t)=\EE\left(\mathbbm{1}_{\left]-\infty,t\right]}\left(\sum_{i=1}^{\lambda(r)} \frac{X_i^{(r)}}{\sigma_r}\right)-\mathbbm{1}_{\left]-\infty,t\right]}(Y)\right).\]
The idea is thus to find, for $t\in\RR$, a family of thrice differentiable function $(h_{t,\varepsilon})_{\varepsilon>0}$ with, for every $\varepsilon>0$, $||h_{t,\varepsilon}^{'''}||_{\infty}<\infty$ and which converges pointwise to the indicator funtion $\mathbbm{1}_{\left]-\infty,t\right]}$ when $\varepsilon$ tends to $0$.

        \subsubsection{Approximation of the indicator function}
There exists a function $f:\RR\rightarrow\RR\in \mathcal{C}^{3}(\RR)$ satisfying the following conditions $f^{'}(0)=f^{'}(1)=f^{''}(0)=f^{''}(1)=f^{'''}(0)=f^{'''}(1)=0$, $f(t)=1$ if $t\le 0$, $f(t)=0$ if $t\ge 1$ and $0\le f(t)\le 1$, for all real $t$.
Then we define the linear function $\theta_{t,\varepsilon}:[t-\varepsilon,t+\varepsilon]\rightarrow [0,1]$, $u\mapsto \frac{1}{2\varepsilon}(\varepsilon-t+u)$. Finally, we get our approximation by 
        \[h_{t,\varepsilon}(u):=\left\{\begin{array}{ll}
         1 & \mbox{if } u\le t-\varepsilon, \\
         f\circ\theta_{t,\varepsilon}(u)& \mbox{if } t-\varepsilon\le u\le t+\varepsilon, \\
         0 & \mbox{otherwise.}
    \end{array}\right.
    \]
\begin{figure}[H]
    \centering
    \includegraphics[scale=0.35]{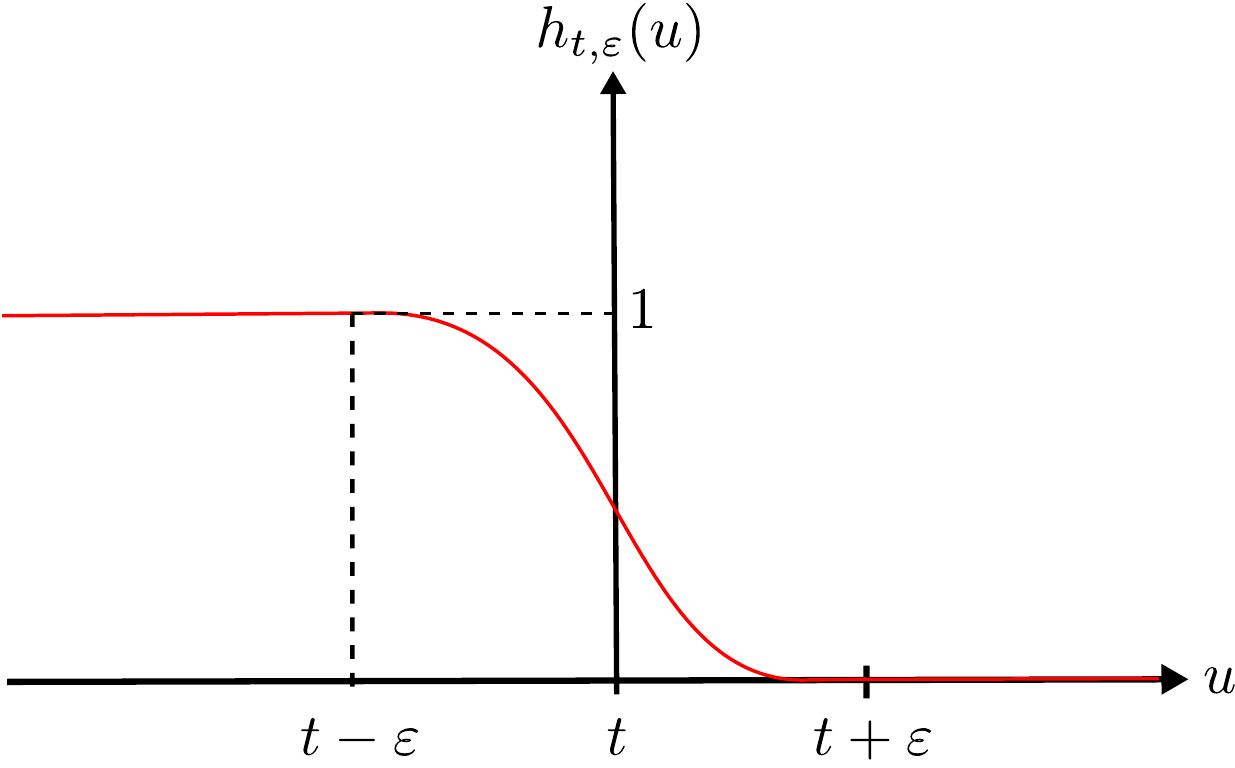}
    \caption{Graph of $h_{t,\varepsilon}$.}
\end{figure}

We have the following properties satisfy by $h_{t,\varepsilon}$.
\begin{lemme}\label{prop de h}
Let $r\ge 1$, let $Y$ be a standard normal random variable.
\begin{enumerate}
    \item $\forall t\in\RR$, the sequence of $\mathcal{C}^3(\RR)$ functions $(h_{t,\varepsilon})_{\varepsilon}$ converges pointwise to the indicator function $\mathbbm{1}_{]-\infty;t]}$ when $\varepsilon$ tends to $0$.
    \item $\forall \varepsilon>0$, $\forall t\in\RR$,  $||h_{t,\varepsilon}^{'''}||_{\infty}=\frac{||f^{'''}||_{\infty}}{8\varepsilon^3}$ and, in particular, the upper bound is independent of $t$.
    \item $\forall\varepsilon>0$, we have, for any random variable $X$
    \begin{align}\label{erreur des fcts de repart}
        \sup_{t\in\RR}\bigg|\PP\left(X\le t\right)-\PP\left (Y\le t \right)\bigg|    &  \le\sup_{t\in\RR}\bigg|\EE\left(h_{t,\varepsilon}\left(X\right)-h_{t,\varepsilon}\left(Y\right)\right)\bigg| \notag \\
        & \quad +\frac{4\varepsilon}{\sqrt{2\pi}}.
    \end{align}
\end{enumerate}
\end{lemme}
The proof of this lemma is given at the end of this paper.
    \subsubsection{Last step of the proof of (\ref{vitesse indicatrice})}
 
\begin{proof}[Proof of \eqref{vitesse indicatrice}]
Let $\varepsilon>0$. From \eqref{majo_fonction_reg} and \eqref{erreur des fcts de repart}, we obtain
\begin{align*}
    \sup_{t\in\RR}\left|\PP\left(\frac{\Delta^{(r)}}{\sigma_r}\le t\right)-\PP\left (Y\le t \right)\right|    & \le\sup_{t\in\RR}\left|\EE\left(h_{t,\varepsilon}\left(\frac{\Delta^{(r)}}{\sigma_r}\right)-h_{t,\varepsilon}\left(Y\right)\right)\right|+\frac{4\varepsilon}{\sqrt{2\pi}}\\
    &\le \sup_{t\in\RR}\left(\frac{K||h_{t,\varepsilon}'''||_{\infty}}{\sqrt{\rho(r)}}\right)+\frac{4\varepsilon}{\sqrt{2\pi}}\\
\end{align*}
Lemma \ref{prop de h} gives $||h_{t,\varepsilon}^{'''}||_{\infty}=\dfrac{||f^{'''}||_{\infty}}{8\varepsilon^3}$. So, we obtain
\[\sup_{t\in\RR}\left|\PP\left(\frac{\Delta^{(r)}}{\sigma_r}\le t\right)-\PP\left (Y\le t \right)\right|\le \frac{||f^{'''}||_{\infty}K}{8\varepsilon^3\sqrt{\rho(r)}}+\frac{4\varepsilon}{\sqrt{2\pi}}\]
Now, we choose $\varepsilon>0$ such that $\dfrac{1}{\varepsilon^3\sqrt{\rho(r)}}= \dfrac{4\varepsilon}{\sqrt{2\pi}}$ 
and get the existence of a constant $\widetilde{K}>0$ such that
\[\sup_{t\in\RR}\left|\PP\left(\frac{\Delta^{(r)}}{\sigma_r}\le t\right)-\PP\left (Y\le t \right)\right|\le\frac{\widetilde{K}}{\rho(r)^{\frac{1}{8}}}.\]
\end{proof}
It only remains to prove Lemma~\ref{prop de h}.

        \subsubsection{Proof of Lemma~\ref{prop de h}}
\begin{proof}[Proof of Lemma~\ref{prop de h}]
Let $\varepsilon>0$. The first point is trivial by construction of $h_{t,\varepsilon}$. The second point is also quite simple to show. 
We write
\begin{align*}
    \sup_{u\in\RR} \left|h_{t,\varepsilon}^{'''}(u)\right| &=\sup_{t-\varepsilon\le u\le t+\varepsilon}\left|h_{t,\varepsilon}^{'''}(u)\right| \\
            &= \sup_{t-\varepsilon\le u\le t+\varepsilon} \left|\theta_{t,\varepsilon}^{'}(u)\right|^3\left|f^{'''}\circ\theta_{t,\varepsilon}(u)\right|\\
            &=\frac{1}{8\varepsilon^3}\sup_{0\le u\le 1} \left|f^{'''}(u)\right|.
\end{align*}
For the last point, let $t$ and $x\in\RR$. Since 
\[h_{t-\varepsilon,\varepsilon}(x)\le \mathbbm{1}_{]-\infty,t]}(x)\le h_{t+\varepsilon,\varepsilon}(x),\]
we deduce that for any random variable $X$
\begin{align}\label{premiere ineg}
    \EE\left(h_{t-\varepsilon,\varepsilon}\left(X\right)\right)\le \PP\left(X\le t\right)\le \EE\left(h_{t+\varepsilon,\varepsilon}\left(X\right)\right).
\end{align}
Moreover, 
\[\EE\left(h_{t+\varepsilon,\varepsilon}(Y)\right)-\EE\left(h_{t-\varepsilon,\varepsilon}(Y)\right)=\int_{t-2\varepsilon}^{t+2\varepsilon}\underbrace{\left(h_{t+\varepsilon,\varepsilon}(y)-h_{t-\varepsilon,\varepsilon}(y)\right)}_{\le 1}\frac{e^{\frac{-y^2}{2}}}{\sqrt{2\pi}}\mathrm{d}y\le \frac{4\varepsilon}{\sqrt{2\pi}}.\]
Hence, we get that
\begin{align}\label{quatrieme ineg}
    \EE\left(h_{t+\varepsilon,\varepsilon}(Y)\right)-\frac{4\varepsilon}{\sqrt{2\pi}}\le \PP\left(Y\le t\right)\le \EE\left(h_{t-\varepsilon,\varepsilon}(Y)\right)+\frac{4\varepsilon}{\sqrt{2\pi}}.
\end{align}
Then, we subtract \eqref{quatrieme ineg} from \eqref{premiere ineg} and we take the supremum over $t\in\RR$ , observing that 
\[\sup_{t\in\RR} \bigg|\EE\bigg(h_{t-\varepsilon,\varepsilon}\left(X\right)-h_{t-\varepsilon,\varepsilon}(Y)\bigg)\bigg|=\sup_{t\in\RR} \bigg|\EE\bigg(h_{t+\varepsilon,\varepsilon}\left(X\right)-h_{t+\varepsilon,\varepsilon}(Y)\bigg)\bigg|.\]
We get \eqref{erreur des fcts de repart}.
\end{proof}

\bibliographystyle{plain}
\bibliography{sample}  

\end{document}